\documentclass[12pt,a4paper, oneside,reqno]{article}

\usepackage[utf8]{inputenc}
\usepackage[english]{babel}

\usepackage[left=1in,right=1in,top=1in,bottom=1in]{geometry}

\usepackage{amsfonts,amssymb,amsmath,amsthm}

\usepackage[symbol]{footmisc} 
\usepackage{longtable}
\usepackage{enumerate}
\usepackage{hyperref}
\usepackage{tkz-graph,comment}
\usepackage{fancyhdr} 
\theoremstyle{plain}
\newtheorem{theorem}{Theorem}
\newtheorem{lemma}[theorem]{Lemma}

\newtheorem{remark}[theorem]{Remark}

\newtheorem{definition}[theorem]{Definition}
\newtheorem{corollary}[theorem]{Corollary}

\usetikzlibrary{arrows}

\newtheorem*{theoremA1}{Theorem A1}
\newtheorem*{theoremA2}{Theorem A2}
\newtheorem*{theoremG1}{Theorem G1}
\newtheorem*{theoremG2}{Theorem G2}

\usepackage{fouriernc} 

\usepackage{csquotes}%

\usepackage[%
backend=biber,%
bibencoding=utf8,%
sorting=nty,
citestyle = ext-numeric-comp, 
bibstyle = ext-numeric-comp,
articlein=false,
language=autobib,%
defernumbers=true,%
sortcites=true,%
doi=false,%
isbn=false,%
hyperref=true,
maxalphanames = 1,
giveninits = true, 
uniquename=init,
abbreviate = true,
bibencoding = utf8,
dateabbrev = true,
maxbibnames = 10,
minbibnames = 3,
url=false
]{biblatex}

\DeclareFieldFormat{pages}{#1}

\begin{document}


\bigskip

\noindent{\Large
The algebraic and geometric classification of \\
right alternative and semi-alternative  algebras}
 \footnote{
The first part of the  work is supported by 
RSF  19-71-30002;
the second part of the work is supported by FCT   2023.08031.CEECIND, UIDB/00212/2020 and UIDP/00212/2020.}

 \bigskip

\begin{center}

 {\bf
Hani Abdelwahab\footnote{Department of Mathematics, 
 Mansoura University,  Mansoura, Egypt; \ haniamar1985@gmail.com}, 
   Ivan Kaygorodov\footnote{CMA-UBI, University of  Beira Interior, Covilh\~{a}, Portugal; \    kaygorodov.ivan@gmail.com}   \&
   Roman Lubkov\footnote{Department of Mathematics and Computer Science, Saint Petersburg University, Russia; r.lubkov@spbu.ru, romanlubkov@yandex.ru}  
}

\end{center}

\noindent {\bf Abstract:}
{\it  
The algebraic and geometric classifications of complex 
$3$-dimensional right alternative and semi-alternative  algebras are given.
As corollaries, we have the algebraic and geometric classification of 
complex $3$-dimensional $\mathfrak{perm}$, binary $\mathfrak{perm}$, associative, $(-1,1)$-, 
binary $(-1,1)$-, and assosymmetric algebras.
In particular, 
we proved that 
the first example of non-associative right alternative   algebras appears in dimension $3;$
the first example of non-associative assosymmetric  algebras appears in dimension $3;$
the first example of non-assosymmetric semi-alternative algebras appears in dimension $4;$
the first example of binary $(-1,1)$-algebras, which is non-$(-1,1)$-, appears in dimension $4;$
the first example of right alternative algebras, which is not binary $(-1,1)$-, appears in dimension $4;$
 the first example of binary $\mathfrak{perm}$ non-$\mathfrak{perm}$ algebras appears in dimension $4.$  
As a byproduct, we give a more easy answer to problem 2.109 from the Dniester Notebook, 
previously resolved by Shestakov and Arenas.}

 \bigskip 


\bigskip


	 \medskip

 
\tableofcontents

\newpage

\section*{Introduction}

The algebraic classification (up to isomorphism) of algebras of dimension $n$ of a certain variety
defined by a family of polynomial identities is a classic problem in the theory of non-associative algebras, see \cite{ikp20, akks, afm, G62, ikm,ikm21}.
There are many results related to the algebraic classification of small-dimensional algebras in different varieties of
associative and non-associative algebras.
For example, algebraic classifications of 
$2$-dimensional algebras,
$3$-dimensional Poisson algebras \cite{afm},
$4$-dimensional alternative  algebras \cite{G62},
$4$-dimensional nilpotent Poisson  algebras,
$4$-dimensional nilpotent right alternative  algebras \cite{ikm},
$5$-dimensional symmetric Leibniz algebras  and so on
 have been given.
 Deformations and geometric properties of a variety of algebras defined by a family of polynomial identities have been an object of study since the 1970's, see \cite{ben,BC99,GRH, GRH3,ikm,mr}  and references in \cite{k23,l24,MS}. 
 Burde and Steinhoff constructed the graphs of degenerations for the varieties of    $3$-dimensional and $4$-dimensional Lie algebras~\cite{BC99}. 
 Grunewald and O'Halloran studied the degenerations for the variety of $5$-dimensional nilpotent Lie algebras~\cite{GRH}.

One of the most important generalizations of associative algebras is the class of alternative algebras.
They are defined by the following identities 
\begin{equation}    \label{ident}
    \big(x,y,z\big) \ = \ - \ \big(x,z,y\big)\mbox{ \ and \ }\big(x,y,z\big) \ = \ \big(y,z,x\big),
\end{equation}
where $\big(x,y,z\big)=(xy)z-x(yz).$
The main example of (non-associative) alternative algebras is the octonion algebra discovered in  the XIX century.
The variety of alternative algebras is under certain consideration now (see, for example, \cite{C24,C25,G24,AA,hnt} and reference therein).
The variety of algebras defined as the first identity from \eqref{ident} is called right alternative algebras and the variety of algebras defined by the second identity from \eqref{ident} is called semi-alternative\footnote{While this particular identity has been studied by various authors and dates back to a 1933 article by Max Zorn \cite{Z}, the term "semialternative'' seems first to have been used by   Schafer in 1992 \cite{Schafer92}. It is also called
weakly alternative algebras in 1993 \cite{V93} and assocyclic algebras in 2011 \cite{AS}.} algebras\footnote{On the other hand, the variety of semi-alternative algebras is a generalization of another variety of algebras, called assosymmetric algebras (i.e., algebras with the following identities $(x,y,z)=(y,x,z)=(x,z,y)$), introduced by Kleinfeld in 1957 \cite{K57} and studied in a series of papers by Dzhumadildaev (see, for example, \cite{DZA} and references in \cite{ikm21}). 
}.
Obviously, each of these varieties is a generalization of alternative (and more general, associative) algebras.

Albert proved that every semisimple right alternative algebra of characteristic not two is alternative \cite{alb2} and 
after that, 
Thedy showed that a finite dimensional right alternative algebra without a nil ideal is alternative \cite{thedy77}.
But, unfortunately, it was proven that the variety of right alternative algebras does not admit the  Wedderburn principal theorem \cite{thedy78}.
In the infinite case we have a different result:
Mikheev constructed an example of an infinite-dimensional simple right alternative algebra that is not alternative over any field \cite{mikh} and 
it  was finally established that a simple right alternative algebra must be either alternative or nil \cite{skos1}.
The study of simple right alternative superalgebras has a very big progress in recent papers of Pchelintsev and Shashkov (see,
for example, see their survey \cite{seroleg} and the reference therein).
Right alternative algebras with some additional identities were studied in papers of Isaev, Pchelintsev and others \cite{isaev}.
For example, Isaev gave the negative answer for the Specht problem in the variety of right alternative algebras \cite{isaev}.
Some combinatorial properties of the variety of right alternative algebras were studied by Umirbaev in \cite{ualbay85}.
Right alternative unital bimodules over the matrix algebras of order $\geq 3$ and over Cayley algebra were studied by 
Murakami,  Pchelintsev,  Shashkov and Shestakov in \cite{MPS,PSS}.  
On the other side, there is an  interest in the study of 
nilpotent, right nilpotent, and solvable right alternative algebras (see, for example, \cite{skos3,ser76,ser13}).
So, Pchelintsev proved an analog of the well-known Zhevlakov theorem: any right alternative Malcev-admissible  nil algebra of bounded index over a field of characteristic zero is solvable \cite{sergey};
Kaygorodov and Popov proved an analog of Moens theorem:
a finite-dimensional right alternative algebra over a field of characteristic zero  admitting an invertible Leibniz-derivation is right nilpotent \cite{kp16}.

The systematic study of semi-alternative algebras from the algebraic point of view started after a paper by Outcalt, where he proved that 
each simple or primitive semi-alternative ring must be  alternative \cite{065}.
Sterling proved that if a semi-alternative ring contains no non-zero ideal whose square is zero,  then it is alternative \cite{S68}.
Some radical properties of semi-alternative rings were discussed by Pokrass \cite{P78}.
Kleinfeld proved that  rings which satisfy 
\begin{center}$\big(x,x,y\big)\ =\ \big(x,y,x\big)\ =\ \big(y,x,x\big)$ \end{center}
must be semi-alternative \cite{K88}.
Kleinfeld and   Widmer \cite{KW} establish that semi-alternative rings  also satisfy the identities
\begin{center}
    $2\big(x,x,x\big)^2\ =\ 0$ and $\big(\big(y,x,x\big),x,x\big)\ =\ 0$.
\end{center}
Schafer proved that each solvable semi-alternative algebra is nilpotent and gave 
an example showing that the Wedderburn principal
theorem  does not hold for semi-alternative algebras   \cite{S68}.
Vakhitov proved that each semi-alternative nil-ring is solvable \cite{V93} and after that
he described semi-alternative modules over composition algebras in \cite{V94}.
One-generator free semialternative rings over $\mathbb Z$ were studied by Bremner \cite{B99}.
The commutator algebra of each semi-alternative algebra gives a binary Lie algebra, 
but as it was proved in a paper by Arenas and Shestakov, there are exceptional binary Lie algebras \cite{AS}.
 
Gainov proved that each $3$-dimensional alternative algebra is associative \cite{G62}. 
The aim of our present work is to classify all $3$-dimensional   right alternative and semi-alternative algebras. 
Based on the obtained results, we will construct the geometric classification of the mentioned varieties of algebras. 
Namely, the algebraic  classification  of complex 
$3$-dimensional right alternative algebras is given in Theorem A1 and semi-alternative  algebras is given in Theorem A2.
As corollaries, we have the algebraic classification of 
complex $3$-dimensional $\mathfrak{perm}$ (Corollary \ref{perm})), \ 
binary $\mathfrak{perm}$ (Corollary \ref{bperm}),  \ 
associative (Theorem \ref{3-dim assoc}), \ 
$(-1,1)$- (Corollary \ref{(-1,1)}), \ 
binary $(-1,1)$- (Corollary \ref{3dimbin-11}), \ 
and assosymmetric (Corollary \ref{assos}) algebras.
In particular, we proved that the first example of non-associative right alternative   algebras appears in dimension $3$; the first example of non-associative assosymmetric  algebras appears in dimension $3;$
the first example of non-assosymmetric semi-alternative algebras appears in dimension $4$ (Theorem \ref{4-dim non-asso});
the first example of binary $(-1,1)$-algebras, which is non-$(-1,1)$-, appears in dimension $4$ (Theorem \ref{non (-1,1)});
the first example of right alternative algebras, which is not binary $(-1,1)$-, appears in dimension $4$ (Corollary \ref{21});
 the first example of binary $\mathfrak{perm}$ non-$\mathfrak{perm}$ algebras appears in dimension $4$
 (Corollary \ref{21}).
As a byproduct, we give an easier answer to problem 2.109 from the Dniester Notebook, 
previously resolved by Shestakov and Arenas (Lemma \ref{commutator}). 
The geometric classifications of complex 
$3$-dimensional 
$\mathfrak{perm}$ (Theorem \ref{thm:geo_perm}), \ 
associative (Theorem \ref{thm:geo_assoc}),\ 
right alternative (Theorem G1), and 
semi-alternative (Theorem G2)  algebras are also  given.

\newpage
\section{The algebraic classification of  algebras}

\subsection{Preliminaries: the algebraic   classification}
All the algebras below will be over $\mathbb C$ and all the linear maps will be $\mathbb C$-linear.
For simplicity, every time we write the multiplication table of an algebra 
the products of basic elements whose values are zero or can be recovered from the commutativity  or from the anticommutativity are omitted.
The notion of a nontrivial algebra means that the  multiplication is nonzero.
In this section, we introduce the techniques used to obtain our main results (the techniques are similar to those considered  in \cite{afm}).

Let $\big({\rm A}, \cdot\big)$ be an algebra. We consider the following two new products
on the underlying vector space ${\rm A}$ defined by 
\begin{center}
$x\circ y :=\frac{1}{2}\big(x\cdot y+y\cdot x\big), \hspace{2cm} [x,y] :=\frac{1}{2}%
\big(x\cdot y-y\cdot x\big).$
\end{center}
Let us denote ${\rm A}^+:= \big({\rm A}, \circ\big),$ \ 
${\rm A}^-:=\big({\rm A},
[\cdot,\cdot]\big)$ 
and the associator as $\big(x,y,z\big)_\cdot:=\big(x\cdot y\big)\cdot z-x\cdot \big(y\cdot z\big).$

Recall that an algebra $\big({\rm A}, \cdot\big)$ is called a  right alternative (resp., semi-alternative) algebra if it satisfies  the identity: 
\begin{center}
  $\big(x,y,z\big) \ = \ - \ \big(x,z,y\big)$ \ $\big($   resp., \  $\big(x,y,z\big) \ = \ \big(y,z,x\big)$  $\big).$ 
\end{center}

\begin{definition}
Let $\big({\rm A},\circ\big)$ be a Jordan algebra. 
Let ${\rm Z}^2\big({\rm A},%
{\rm A}\big)$ be the set of all skew-symmetric  bilinear maps $\theta :%
{\rm A}\times {\rm A} \to {\rm A}$ such that 
\begin{longtable}{lc}
$\big(x, y,  z\big)_{\circ}+ \theta\big(x,y\big) \circ z + \theta\big(x \circ y, z\big) -  
  x \circ \theta \big(y , z\big) - \theta\big(x, y \circ  z\big) + \big(x,y,z\big)_{\theta}$ & $=$\\
\multicolumn{2}{r}{ $-\big(\big(x, z,  y\big)_{\circ}+ \theta\big(x,z\big) \circ y + \theta\big(x \circ z, y\big) -  
  x \circ \theta \big(z , y\big) - \theta\big(x, z \circ  y\big) + \big(x,z,y\big)_{\theta} \big).$}

\end{longtable}  
\end{definition}

\begin{definition}\label{Z2}
Let $\big({\rm A},[\cdot,\cdot]\big)$ be a Malcev algebra. 
Let ${\rm Z}^2\big({\rm A},%
{\rm A}\big)$ be the set of all symmetric  bilinear maps $\theta :%
{\rm A}\times {\rm A} \to {\rm A}$ such that 
\begin{longtable}{lc}
$\big(x, y,  z\big)_{[\cdot,\cdot]}+ [\theta\big(x,y\big) , z] + \theta\big([x,  y], z\big) -  
  [x , \theta \big(y , z\big)] - \theta\big(x, [y ,  z]\big) + \big(x,y,z\big)_{\theta}$ & $=$\\
\multicolumn{2}{r}{ $\big(y, z,  x\big)_{[\cdot,\cdot]}+ [\theta\big(y,z\big) , x] + \theta\big([y, z], x\big) -  
  [y , \theta \big(z , x\big)] - \theta\big(y, [z ,  x]\big) + \big(y,z,x\big)_{\theta} .$}
\end{longtable}  
\end{definition}

\noindent  For $\theta \in {\rm Z}^2\big(%
{\rm A},{\rm A}\big)$ we define on ${\rm A}$ a product $%
*_{\theta} :{\rm A}\times {\rm A}\to {\rm A}$ by 
$
x *_{\theta} y := \theta(x,y).  $

\begin{lemma}
Let $({\rm A},\cdot)$ be a Jordan  (resp., Malcev) algebra and $\theta \in {\rm Z}^2({\rm A},{\rm A})$. Then  $( {\rm A},\cdot_{\theta})$ is a
right alternative (resp., semi-alternative)  algebra, 
where $x\cdot_{\theta} y := x \cdot y + x *_{\theta} y.$ 

\end{lemma}

Now, let $\big({\rm A},\cdot\big)$ be an algebra and ${\rm Aut}\big(%
{\rm A}\big)$ be the automorphism group of ${\rm A}$ with respect to
product $\cdot$. Then ${\rm Aut}\big({\rm A}\big)$ acts on ${\rm Z}^2\big({\rm A},%
{\rm A}\big)$ by 
\begin{equation*}
\big(\theta \ast \varphi\big)(x,y) := \varphi^{-1}\big(\theta \big(\varphi(x),\varphi(y) \big)
\big),
\end{equation*}
where $\varphi \in {\rm Aut\big({\rm A}\big)}$ and $\theta \in {\rm Z}^2\big({\rm A}%
, {\rm A}\big)$.

\begin{lemma}
Let $\big({\rm A},\cdot\big)$ be a Jordan (resp., Malcev) algebra and $\theta, \vartheta \in {\rm Z}^2\big(%
{\rm A},{\rm A}\big)$. Then the   algebras $\big({\rm A}, \cdot_{\theta}\big)$ and $\big({\rm A}, \cdot_{\vartheta}\big)$ are isomorphic if and only if there exists $\varphi \in {\rm Aut\big({\rm A}\big)}$ satisfying $\theta \ast \varphi
=\vartheta $.
\end{lemma}

Hence, we have a procedure to classify the right alternative (resp., semi-alternative) algebras associated with a given Jordan (resp., Malcev) algebra $\big({\rm A},\cdot\big)$. It consists of three steps:

\begin{enumerate}
\item[{\bf Step} $1.$] Compute ${\rm Z}^2\big({\rm A},{\rm A}\big)$.

\item[{\bf Step} $2.$] Find the orbits of ${\rm Aut\big({\rm A}\big)}$ on ${\rm Z}^2\big(%
{\rm A},{\rm A}\big)$.

\item[{\bf Step} $3.$] Choose a representative $\theta$ from each orbit and then
construct the right alternative (resp., semi-alternative)   algebra $\big({\rm A}, \cdot_{\theta}\big)$.
\end{enumerate}

Let us introduce the following notations. Let $\{e_1,\dots,e_n\}$ be a fixed basis of an algebra $\big({\rm A},\cdot\big)$. Define 
$\mathrm{\Lambda }^2({\rm A},\mathbb{C})$ to be the space of all 
skew-symmetric  (resp., symmetric) bilinear forms on ${\rm A}$, that is, \begin{center}
$\mathrm{\Lambda}^2(%
{\rm A},\mathbb{C}) := 
\langle \Delta_{ij} | 1\leq i<j \leq n\rangle$
\ \big( \ resp., $\mathrm{\Lambda}^2(%
{\rm A},\mathbb{C}) := 
\langle \Delta_{ij} | 1\leq i\leq j \leq n\rangle$ \ \big),
\end{center} where $\Delta_{ij}$ is the skew-symmetric (resp., symmetric) bilinear form $%
\Delta_{ij}:{\rm A}\times {\rm A} \to \mathbb{C}$ defined by 
\begin{equation*}
\Delta_{ij}( e_l,e_m) :=\left\{ 
\begin{tabular}{rl}
$1,$ & if $(i,j) = (l,m),$ \\ 
$(-1)^\varepsilon,$ & if $(i,j) = (m,l)$ and $\varepsilon=1$  (resp., $\varepsilon=0$), \\ 
$0,$ & otherwise.%
\end{tabular}
\right.
\end{equation*}
Now, if $\theta \in {\rm Z}^2\big({\rm A},{\rm A}\big)$ then $\theta$ can be
uniquely written as $\theta (x,y) = \sum_{i=1}^n B_i(x,y)e_i$, where $%
B_1,\dots , B_n$ are skew-symmetric  (resp., symmetric) bilinear forms on ${\rm A}$. Also,
we may write $\theta = \big(B_{1},\dots ,B_n \big)$. Let $\varphi^{-1}\in {\rm Aut}(%
{\rm A})$ be given by the matrix $( b_{ij})$. If $(\theta \ast
\varphi)(x,y) = \sum_{i=1}^n B_i^{\prime }(x,y)e_i$, then $B_i^{\prime }=
\sum_{j=1}^n b_{ij}\varphi^t B_j\varphi$, whenever $i \in \{1,\dots,n\}$.

\begin{theorem}
\label{3-dim commassoc}Let ${\rm A}$ be a complex $3$-dimensional
associative commutative algebra. Then ${\rm A}$ is isomorphic to one of
the following algebras:

\begin{longtable}{llllll}
${\rm A}_{01}$ & $:$ & $e_{1}e_{1} = e_{1}$ & $e_{2}e_{2} = e_{2}$ & \\ 
${\rm A}_{02}$ & $:$ & $e_{1}e_{1} = e_{1}$ & $e_{1}e_{2} = e_{2}$ & \\ 
${\rm A}_{03}$ & $:$ & $e_{1}e_{1} = e_{1}$ & & \\ 
${\rm A}_{04}$ & $:$ & $e_{1}e_{1} = e_{2}$ & & \\ 
${\rm A}_{05}$ & $:$ & $e_{1}e_{2} = e_{3}$ & & \\ 
${\rm A}_{06}$ & $:$ & $e_{1}e_{1} = e_{2}$ & $e_{1}e_{2} = e_{3}$ & \\ 
${\rm A}_{07}$ & $:$ & $e_{1}e_{1} = e_{1}$ & $e_{2}e_{2} = e_{2}$ & $e_{3}e_{3} = e_{3}$ \\ 
${\rm A}_{08}$ & $:$ & $e_{1}e_{1} = e_{1}$ & $e_{2}e_{2} = e_{2}$ & $e_{2}e_{3} = e_{3}$ \\ 
${\rm A}_{09}$ & $:$ & $e_{1}e_{1} = e_{1}$ & $e_{1}e_{2} = e_{2}$ & $e_{1}e_{3} = e_{3}$ \\ 
${\rm A}_{10}$ & $:$ & $e_{1}e_{1} = e_{1}$ & $e_{1}e_{2} = e_{2}$ & $e_{1}e_{3} = e_{3}$ & $e_{2}e_{2} = e_{3}$ \\ 
${\rm A}_{11}$ & $:$ & $e_{1}e_{1} = e_{1}$ & $e_{2}e_{2} = e_{3}$ & \\ 
\end{longtable}
 
\end{theorem}

\begin{theorem}
\label{3-dim Jord}Let ${\rm A}$ be a complex $3$-dimensional Jordan
algebra. Then ${\rm A}$ is an associative commutative algebra listed in
Theorem \ref{3-dim commassoc} or isomorphic to one of the following algebras:

\begin{longtable}{lllllllllll}
${\rm J}_{12}$ & $:$ & $e_{1}e_{1} = e_{1}$ & $e_{2}e_{2} = e_{2}$ & $e_{3}e_{3} = e_{1} + e_{2}$ & $e_{1}e_{3} = \frac{1}{2}e_{3}$ & $e_{2}e_{3} = \frac{1}{2}e_{3}$ & \\  
${\rm J}_{13}$ & $:$ & $e_{1}e_{1} = e_{1}$ & $e_{1}e_{2} = \frac{1}{2}e_{2}$ & $e_{1}e_{3} = e_{3}$ & & & \\  
${\rm J}_{14}$ & $:$ & $e_{1}e_{1} = e_{1}$ & $e_{1}e_{2} = \frac{1}{2}e_{2}$ & $e_{1}e_{3} = \frac{1}{2}e_{3}$ & & & \\  
${\rm J}_{15}$ & $:$ & $e_{1}e_{1} = e_{1}$ & $e_{2}e_{2} = e_{3}$ & $e_{1}e_{2} = \frac{1}{2}e_{2}$ & & & \\  
${\rm J}_{16}$ & $:$ & $e_{1}e_{1} = e_{1}$ & $e_{2}e_{2} = e_{3}$ & $e_{1}e_{2} = \frac{1}{2}e_{2}$ & $e_{1}e_{3} = e_{3}$ & & \\ 
${\rm J}_{17}$ & $:$ & $e_{1}e_{1} = e_{1}$ & $e_{2}e_{2} = e_{2}$ & $e_{1}e_{3} = \frac{1}{2}e_{3}$ & $e_{2}e_{3} = \frac{1}{2}e_{3}$ & & \\  
${\rm J}_{18}$ & $:$ & $e_{1}e_{1} = e_{1}$ & $e_{1}e_{2} = \frac{1}{2}e_{2}$ & & & & \\  
${\rm J}_{19}$ & $:$ & $e_{1}e_{1} = e_{1}$ & $e_{2}e_{2} = e_{2}$ & $e_{1}e_{3} = \frac{1}{2}e_{3}$ & & & \\ 
\end{longtable}
 
\end{theorem}

\subsection{Right alternative algebras}

\begin{lemma}
Let ${\rm R}$\ be a complex $3$-dimensional right alternative algebra.
Then ${\rm R}$\ is anticommutative if and only if it is $2$-step
nilpotent. So if ${\rm R}$\ is anticommutative, then it\ is isomorphic
to the following algebra:

\begin{longtable}{lcll}
 ${\rm A}$&$:$&$e_{1}e_{2}=e_3$&$e_{2}e_{1}=-e_{3}.$
\end{longtable}
\end{lemma}

\begin{lemma}
Let ${\rm R}$\ be a right alternative algebra. Then

\begin{itemize}
\item ${\rm R}^{+}$ is a Jordan algebra.

\item If ${\rm R}$ is commutative, then ${\rm R}$\ is associative.

 
\end{itemize}
\end{lemma}

\begin{remark}
It is known that if ${\rm A}$ is power-associative, then so is ${\rm A}^{+}$. The converse is not true. For instance, any right alternative
algebra ${\rm R}$ is not necessarily power-associative however ${\rm 
R}^{+}$ is power-associative since ${\rm R}^{+}$ is a Jordan algebra.
\end{remark}

\subsubsection{The classification Theorem A1}

\begin{theoremA1}
\label{3-dim-alt}Let ${\rm R}$\ be a complex $3$-dimensional right
alternative algebra. Then ${\rm R}$\ is  commutative and listed in Theorem \ref{3-dim commassoc} 
or
 it is isomorphic to one of the following algebras:

\begin{longtable}{lllllllll}
${\rm R}_{00}$&$:$&$e_{1}e_{2}=e_3$&$e_{2}e_{1}=-e_{3}$\\
${\rm R}_{01}$ & $:$ & $e_{1}e_{1} = e_{2}$ & $e_{1}e_{3} = e_{2}$ & $e_{3}e_{1} = -e_{2}$ & & & \\ 
${\rm R}_{02}^{\alpha \neq 0}$ & $:$ & $e_{1}e_{2} = \big( 1+\alpha \big) e_{3}$ & $e_{2}e_{1} = \big( 1-\alpha \big) e_{3}$ & & & & \\ 
${\rm R}_{03}$ & $:$ & $e_{1}e_{1} = e_{1}$ & $e_{1}e_{2} = e_{2}$ & $e_{1}e_{3} = e_{3}$ & $e_{3}e_{1} = e_{3}$ & & \\ 
${\rm R}_{04}$ & $:$ & $e_{1}e_{1} = e_{1}$ & $e_{1}e_{2} = e_{2} + e_{3}$  & $e_{1}e_{3} = e_{3}$ & $e_{2}e_{1} = -e_{3}$& $e_{3}e_{1} = e_{3}$ & \\ 
${\rm R}_{05}$ & $:$ & $e_{1}e_{1} = e_{1}$ &   $e_{1}e_{3} = e_{3}$ & $e_{2}e_{1} = e_{2}$ &$e_{3}e_{1} = e_{3}$ & & \\ 
${\rm R}_{06}$ & $:$ & $e_{1}e_{1} = e_{1}$ & $e_{1}e_{2} = e_{2}$ & $e_{1}e_{3} = e_{3}$ & & & \\ 
${\rm R}_{07}$ & $:$ & $e_{1}e_{1} = e_{1}$ & $e_{2}e_{1} = e_{2}$ & $e_{3}e_{1} = e_{3}$ & & & \\ 
${\rm R}_{08}$ & $:$ & $e_{1}e_{1} = e_{1}$   & $e_{1}e_{3} = e_{3}$ & $e_{2}e_{1} = e_{2}$ & & & \\ 
${\rm R}_{09}$ & $:$ & $e_{1}e_{1} = e_{1}$   & $e_{2}e_{1} = e_{2}$ & $e_{2}e_{2} = e_{3}$ & & & \\ 
${\rm R}_{10}$ & $:$ & $e_{1}e_{1} = e_{1}$   & $e_{1}e_{2} = e_{2}$ & $e_{1}e_{3} = e_{3}$ & $e_{2}e_{2} = e_{3}$ & $e_{3}e_{1} = e_{3}$ & \\ 
${\rm R}_{11}$ & $:$ & $e_{1}e_{1} = e_{1}$   & $e_{1}e_{3} = e_{3}$ & $e_{2}e_{2} = e_{2}$ & $e_{3}e_{2} = e_{3}$ & & \\ 
${\rm R}_{12}$ & $:$ & $e_{1}e_{1} = e_{1}$ & $e_{2}e_{1} = e_{2}$ & & & & \\ 
${\rm R}_{13}$ & $:$ & $e_{1}e_{1} = e_{1}$ & $e_{1}e_{2} = e_{3}$ & $e_{2}e_{1} = e_{2} - e_{3}$ & & & \\ 
${\rm R}_{14}$ & $:$ & $e_{1}e_{1} = e_{1}$ & $e_{1}e_{2} = e_{2}$ & & & & \\ 
${\rm R}_{15}$ & $:$ & $e_{1}e_{1} = e_{1}$  & $e_{1}e_{3} = e_{3}$ & $e_{2}e_{2} = e_{2}$ & & & \\ 
${\rm R}_{16}$ & $:$ & $e_{1}e_{1} = e_{1}$ & $e_{2}e_{2} = e_{2}$ & $e_{3}e_{1} = e_{3}$ & & & \\ 
\end{longtable}

\noindent All listed algebras are non-isomorphic except: ${\rm R}_{02}^{\alpha
}\cong {\rm R}_{02}^{-\alpha }$.
\end{theoremA1}

\subsubsection{Proof of Theorem A1}

If ${\rm R}^{+}\in \big\{ {\rm A}_{01},{\rm A}_{02},{\rm A}%
_{03}\big\} \cup \big\{ {\rm A}_{06},\mathcal{\ldots },{\rm A}%
_{11}\big\} $, then ${\rm Z}^{2}\big( {\rm R}^{+},{\rm R}^{+}\big)
=\{0\}$. Further  ${\rm Z}^{2}\big( 
{\rm J}_{12},{\rm J}_{12}\big) =\varnothing $. So we study the following
cases:

\begin{enumerate}[I.]
    \item 

\underline{${\rm R}^{+}={\rm A}_{04}$.}\ Let $\theta \ =\ \big(
B_{1},B_{2},B_{3}\big) \neq 0$ be an arbitrary element of ${\rm Z}^{2}\big( 
{\rm A}_{04},{\rm A}_{04}\big) $. Then 
$\theta \ =\ \big( 0,\ \alpha\Delta _{13}, \ 0\big) $ for some $\alpha \in \mathbb{C}^{\ast }$. The
automorphism group of ${\rm A}_{04}$, ${\rm Aut}\big( {\rm A}_{04}\big) $, consists of the automorphisms $\varphi $ given by a matrix of
the following form:%
\begin{equation*}
\varphi =%
\begin{pmatrix}
a_{11} & 0 & 0 \\ 
a_{21} & a_{11}^{2} & a_{23} \\ 
a_{31} & 0 & a_{33}%
\end{pmatrix}%
.
\end{equation*}%
Let $\varphi =\bigl(a_{ij}\bigr)\in $ ${\rm Aut}\big( {\rm A}_{04}\big) 
$. Then $\theta \ast \varphi =\big( 0, \ \beta \Delta _{13},\ 0\big) $ where $%
\beta =\alpha {a_{33}}{a^{-1}_{11}}$. Let $\varphi $\ be the following
automorphism: 
\begin{equation*}
\varphi =%
\begin{pmatrix}
\alpha & 0 & 0 \\ 
0 & \alpha ^{2} & 0 \\ 
0 & 0 & 1%
\end{pmatrix}%
.
\end{equation*}%
Then $\theta \ast \varphi =\big( 0,\ \Delta _{13},\ 0\big) $. So we get the
algebra ${\rm R}_{01}$.

\item
\underline{${\rm R}^{+}={\rm A}_{05}$.}\ Let $\theta \ =\ \big(
B_{1},B_{2},B_{3}\big) \neq 0$ be an arbitrary element of ${\rm Z}^{2}\big( 
{\rm A}_{05},{\rm A}_{05}\big) $. Then 
$\theta \ =\ \big( 0, \ 0,\ \alpha \Delta _{12}\big) $ for some $\alpha \in \mathbb{C}^{\ast }$. The
automorphism group of ${\rm A}_{05}$, ${\rm Aut}\big( {\rm A}%
_{05}\big) $, consists of the automorphisms $\varphi $ given by a matrix of
the following form:%
\begin{equation*}
\varphi =%
\begin{pmatrix}
\zeta a_{11} & \xi a_{12} & 0 \\ 
\xi a_{21} & \zeta a_{22} & 0 \\ 
a_{31} & a_{32} & \zeta a_{11}a_{22}+\xi a_{12}a_{21}%
\end{pmatrix}%
:\ 
\big( \zeta ,\xi \big) \in \big\{( 1,0) ,(
0,1) \big\}.
\end{equation*}%
Let $\varphi =\bigl(a_{ij}\bigr)\in $ ${\rm Aut}\big( {\rm A}_{05}\big) 
$. Then $\theta \ast \varphi =\big( 0,\ 0,\ \beta \Delta _{12}\big) $ where $%
\beta ^{2}=\alpha ^{2}$. So we get the algebras ${\rm R}_{02}^{\alpha
\neq 0}$.

\item
\underline{${\rm R}^{+}={\rm J}_{13}$.} Let $\theta \ =\ \big(
B_{1},B_{2},B_{3}\big) \neq 0$ be an arbitrary element of  ${\rm Z}^{2}\big( {\rm J}_{13},{\rm J}_{13}\big) $. Then $\theta \in \big\{ \eta _{1},\ \eta _{2}\big\} $ where%
\begin{longtable}{lcllcl}
$\eta _{1} $&$=$&$\big( 0,\ \frac{1}{2}\Delta _{12},\ \alpha \Delta _{12}\big),$
&
$\eta _{2} $&$=$&$\big( 0,\ -\frac{1}{2}\Delta _{12},\ 0\big) ,$
\end{longtable}%
for some $\alpha \in \mathbb{C}$. The automorphism group of ${\rm J}_{13} $, ${\rm Aut}\big( {\rm J}_{13}\big) $, consists of the
invertible matrices of the following form:%
\begin{equation*}
\varphi =%
\begin{pmatrix}
1 & 0 & 0 \\ 
a_{21} & a_{22} & 0 \\ 
0 & 0 & a_{33}%
\end{pmatrix}%
.
\end{equation*}

\begin{itemize}
\item $\theta =\eta _{1}$. If $\alpha =0$, we get the algebra ${\rm R}%
_{03}$. If $\alpha \neq 0$, we choose $\varphi $ to be the following
automorphism:%
\begin{equation*}
\varphi =%
\begin{pmatrix}
1 & 0 & 0 \\ 
0 & 1 & 0 \\ 
0 & 0 & \alpha%
\end{pmatrix}%
.
\end{equation*}%
Then $\theta \ast \varphi =\big( 0,\ \frac{1}{2}\Delta _{12},\ \Delta
_{12}\big) $. So we get the algebra ${\rm R}_{04}$.

\item $\theta =\eta _{2}$. We get the algebra ${\rm R}_{05}$.
\end{itemize}

\item
\underline{${\rm R}^{+}={\rm J}_{14}$.} Let $\theta \ =\ \big(
B_{1},B_{2},B_{3}\big) \neq 0$ be an arbitrary element of  ${\rm Z}^{2}\big( {\rm J}_{14},{\rm J}_{14}\big) $. Then 
$\theta \ =\ \big( 0, \ \alpha _{1}\Delta
_{12}+\alpha _{2}\Delta _{13}, \ \alpha _{3}\Delta _{12}+\alpha _{4}\Delta
_{13}\big) $ such that 
\begin{longtable}{rclrcl}
$\alpha _{1}^{2}-\alpha _{4}^{2} $&$=$&$0,$ &
$\alpha _{3}\big( \alpha _{1}+\alpha _{4}\big) $&$=$&$0,$ \\
$\alpha _{2}\big( \alpha _{1}+\alpha _{4}\big) $&$=$&$0,$ &
$4\alpha _{2}\alpha _{3}+4\alpha _{4}^{2} $&$=$&$1,$
\end{longtable}%
for some $\alpha _{1},\ldots ,\alpha _{4}\in \mathbb{C}$. The automorphism
group of ${\rm J}_{14}$, ${\rm Aut}\big( {\rm J}_{14}\big) $,
consists of the invertible matrices of the following form:%
\begin{equation*}
\varphi =%
\begin{pmatrix}
1 & 0 & 0 \\ 
a_{21} & a_{22} & a_{23} \\ 
a_{31} & a_{32} & a_{33}%
\end{pmatrix}%
.
\end{equation*}%
Let $\varphi =\bigl(a_{ij}\bigr)\in {\rm Aut}\big( {\rm J}_{14}\big) $
and write 
\begin{equation*}
\theta \ast \varphi =\big( 0,\ \beta _{1}\Delta _{12}+\beta _{2}\Delta_{13},\ \beta _{3}\Delta _{12}+\beta _{4}\Delta _{13}\big) .
\end{equation*}%
Then 
\begin{longtable}{lcl}
$\beta _{1} $&$=$&${(\det \varphi)^{-1} }\big( \alpha _{1}a_{22}a_{33}+\alpha
_{2}a_{32}a_{33}-\alpha _{3}a_{22}a_{23}-\alpha _{4}a_{23}a_{32}\big) ,$ \\
$\beta _{2} $&$=$&${(\det \varphi)^{-1} }\big( \alpha _{2}a_{33}^{2}-\alpha
_{3}a_{23}^{2}+\alpha _{1}a_{23}a_{33}-\alpha _{4}a_{23}a_{33}\big) ,$ \\
$\beta _{3} $&$=$&$-{(\det \varphi)^{-1} }\big( \alpha _{2}a_{32}^{2}-\alpha
_{3}a_{22}^{2}+\alpha _{1}a_{22}a_{32}-\alpha _{4}a_{22}a_{32}\big) ,$ \\
$\beta _{4} $&$=$&$-{(\det \varphi)^{-1} }\big( \alpha _{1}a_{23}a_{32}+\alpha
_{2}a_{32}a_{33}-\alpha _{3}a_{22}a_{23}-\alpha _{4}a_{22}a_{33}\big) .$
\end{longtable}%
Then%
\begin{equation*}
\begin{pmatrix}
\beta _{1} & \beta _{2} \\ 
\beta _{3} & \beta _{4}%
\end{pmatrix}%
= 
\begin{pmatrix}
a_{22} & a_{23} \\ 
a_{32} & a_{33}%
\end{pmatrix}
  ^{-1}%
\begin{pmatrix}
\alpha _{1} & \alpha _{2} \\ 
\alpha _{3} & \alpha _{4}%
\end{pmatrix}
\begin{pmatrix} 
a_{22} & a_{23} \\ 
a_{32} & a_{33}%
\end{pmatrix}.
\end{equation*}%
From here, we may assume $%
\begin{pmatrix}
\alpha _{1} & \alpha _{2} \\ 
\alpha _{3} & \alpha _{4}%
\end{pmatrix}%
\in \Biggl\{ 
\begin{pmatrix}
\alpha & 0 \\ 
0 & \beta%
\end{pmatrix}%
,%
\begin{pmatrix}
\alpha & 1 \\ 
0 & \alpha%
\end{pmatrix}%
\Biggr\} $.

\begin{itemize}
\item $%
\begin{pmatrix}
\alpha _{1} & \alpha _{2} \\ 
\alpha _{3} & \alpha _{4}%
\end{pmatrix}%
=%
\begin{pmatrix}
\alpha & 0 \\ 
0 & \beta%
\end{pmatrix}%
$. Then $\theta \in {\rm Z}^{2}\big( {\rm J}_{14},{\rm J}_{14}\big) $ if $%
\alpha _{1}^{2}=\alpha _{4}^{2},$ \ $4\alpha _{4}^{2}=1$. So we may assume $\big( \alpha _{1},\alpha _{2},\alpha _{3},\alpha _{4}\big) \in
\big\{ \big( \frac{1}{2},0,0,\frac{1}{2}\big) ,\big( -\frac{1}{2},0,0-\frac{1}{%
2}\big) ,\big( -\frac{1}{2},0,0,\frac{1}{2}\big),\big( \frac{1}{2},0,0,-\frac{1}{2}\big) \big\} $. Moreover, up to isomorphism, if  $\big( \alpha _{1},\alpha _{2},\alpha _{3},\alpha _{4}\big)=\big( -\frac{1}{2},0,0,\frac{1}{2}\big)$ or $\big( \alpha _{1},\alpha _{2},\alpha _{3},\alpha _{4}\big)=\big( \frac{1}{2},0,0,-\frac{1}{2}\big)$, then we get the same algebra. To see this, let $\phi $ be the following automorphism:%
\begin{equation*}
\phi =%
\begin{pmatrix}
1 & 0 & 0 \\ 
0 & 0 & 1 \\ 
0 & 1 & 0%
\end{pmatrix}%
.
\end{equation*}%
Then $\left( 0,-\frac{1}{2}\Delta _{1,2},\frac{1}{2}\Delta _{1,3}\right)
\ast \phi =\left( 0,\frac{1}{2}\Delta _{1,2},-\frac{1}{2}\Delta
_{1,3}\right) $. Hence we get
the algebras ${\rm R}_{06},$ ${\rm R}_{07},$ and  ${\rm R}_{08}$.

\item $%
\begin{pmatrix}
\alpha _{1} & \alpha _{2} \\ 
\alpha _{3} & \alpha _{4}%
\end{pmatrix}%
=%
\begin{pmatrix}
\alpha & 1 \\ 
0 & \alpha%
\end{pmatrix}%
$. Then $\theta \notin {\rm Z}^{2}\big( {\rm J}_{14},{\rm J}_{14}\big) $.
\end{itemize}

\item 
\underline{${\rm R}^{+}={\rm J}_{15}$.} Let $\theta \ =\ \big(
B_{1},B_{2},B_{3}\big) \neq 0$ be an arbitrary element of  $ {\rm Z}^{2}\big({\rm J}_{15}, 
{\rm J}_{15}\big) $. Then 
$\theta \ =\ \big( 0, \ -\frac{1}{2}\Delta_{12},\ \alpha \Delta _{12}\big) $ for some $\alpha \in \mathbb{C}$. The
automorphism group of ${\rm J}_{15}$, ${\rm Aut}\big( {\rm J}_{15}\big) $, consists of the invertible matrices of the following form:%
\begin{equation*}
\varphi =%
\begin{pmatrix}
1 & 0 & 0 \\ 
a_{21} & a_{22} & 0 \\ 
a_{21}^{2} & 2a_{21}a_{22} & a_{22}^{2}%
\end{pmatrix}%
.
\end{equation*}%
Let $\varphi $\ be the following automorphism:%
\begin{equation*}
\varphi =%
\begin{pmatrix}
1 & 0 & 0 \\ 
-\alpha & 1 & 0 \\ 
\alpha ^{2} & -2\alpha & 1%
\end{pmatrix}%
\end{equation*}%
Then $\theta \ast \varphi =\big( 0,\ -\frac{1}{2}\Delta _{12},\ 0\big) $. Hence
we get the algebra ${\rm R}_{09}$.

\item 
\underline{${\rm R}^{+}={\rm J}_{16}$.} Let $\theta \ =\ \big(
B_{1},B_{2},B_{3}\big) \neq 0$ be an arbitrary element of $ {\rm Z}^{2}\big( {\rm J}_{16},%
{\rm J}_{16}\big) $. Then $\theta \ =\ \big( 0, \ \frac{1}{2}\Delta_{12}, \ \alpha \Delta _{12}\big) $ for some $\alpha \in \mathbb{C}$. The
automorphism group of ${\rm A}_{16}$, ${\rm Aut}\big( {\rm A}%
_{16}\big) $, consists of the invertible matrices of the following form:%
\begin{equation*}
\varphi =%
\begin{pmatrix}
1 & 0 & 0 \\ 
a_{21} & a_{22} & 0 \\ 
-a_{21}^{2} & -2a_{21}a_{22} & a_{22}^{2}%
\end{pmatrix}%
.
\end{equation*}%
Let $\varphi $\ be the following automorphism:%
\begin{equation*}
\varphi =%
\begin{pmatrix}
1 & 0 & 0 \\ 
-\alpha & 1 & 0 \\ 
-\alpha ^{2} & 2\alpha & 1%
\end{pmatrix}%
\end{equation*}%
Then $\theta \ast \varphi =\big( 0,\ \frac{1}{2}\Delta _{12},\ 0\big) $. Hence
we get the algebra ${\rm R}_{10}$.

\item 
\underline{${\rm R}^{+}={\rm J}_{17}$.} Let $\theta \ =\ \big(
B_{1},B_{2},B_{3}\big) \neq 0$ be an arbitrary element of ${\rm Z}^{2}\big({\rm J}_{17},{\rm J}_{17}\big) $. Then $\theta \in \big\{ \eta _{1},\ \eta_{2}\big\} $ where%
\begin{longtable}{lcllcl}
$\eta _{1} $&$=$&$\big( 0,\ 0,\ \frac{1}{2}\Delta _{13}-\frac{1}{2}\Delta
_{23}\big) ,$ &
$\eta _{2} $&$=$&$\big( 0, \ 0,\ -\frac{1}{2}\Delta _{13}+\frac{1}{2}\Delta
_{23}\big) .$
\end{longtable}%
The automorphism group of ${\rm J}_{17}$, ${\rm Aut}\big( {\rm J}%
_{17}\big) $, consists of the invertible matrices of the following form:%
\begin{equation*}
\begin{pmatrix}
1 & 0 & 0 \\ 
0 & 1 & 0 \\ 
a_{31} & -a_{31} & a_{33}%
\end{pmatrix}
, \  
\begin{pmatrix}
0 & 1 & 0 \\ 
1 & 0 & 0 \\ 
a_{31} & -a_{31} & a_{33}%
\end{pmatrix}.
\end{equation*}%
Consider the following automorphism:%
\begin{equation*}
\varphi = 
\begin{pmatrix} 
0 & 1 & 0 \\ 
1 & 0 & 0 \\ 
0 & 0 & 1%
\end{pmatrix}.
\end{equation*}%
Then $\eta _{2}\ast \varphi =\eta _{1}$. Hence we get the algebra ${\rm R}%
_{11}$.

\item 
\underline{${\rm R}^{+}={\rm J}_{18}$.} Let $\theta \ =\ \big(
B_{1},B_{2},B_{3}\big) \neq 0$ be an arbitrary element of  $ {\rm Z}^{2}\big( {\rm J}_{18}, {\rm J}_{18}\big) $. Then $\theta \in \big\{ \eta _{1},\ \eta
_{2}\big\} $ where%
\begin{longtable}{lcllcl}
$\eta _{1} $&$=$&$\big( 0,\ -\frac{1}{2}\Delta _{12},\ \alpha \Delta _{12}\big) ,$ &
$\eta _{2} $&$=$&$\big( 0,\ \frac{1}{2}\Delta _{12},\ 0\big) ,$
\end{longtable}%
for some $\alpha \in \mathbb{C}$. The automorphism group of ${\rm J}_{18} $, ${\rm Aut}\big( {\rm J}_{18}\big) $, consists of the
invertible matrices of the following form:%
\begin{equation*}
\varphi =%
\begin{pmatrix}
1 & 0 & 0 \\ 
a_{21} & a_{22} & 0 \\ 
0 & 0 & a_{33}%
\end{pmatrix}%
.
\end{equation*}

\begin{itemize}
\item $\theta =\eta _{1}$. If $\alpha =0$, we get the algebra ${\rm R}%
_{12}$. If $\alpha \neq 0$, we choose $\varphi $ to be the following
automorphism:%
\begin{equation*}
\varphi =%
\begin{pmatrix}
1 & 0 & 0 \\ 
0 & 1 & 0 \\ 
0 & 0 & \alpha%
\end{pmatrix}%
.
\end{equation*}%
Then $\theta \ast \varphi =\big( 0,\ -\frac{1}{2}\Delta _{12},\ \Delta
_{12}\big) $. So we get the algebra ${\rm R}_{13}$.

\item $\theta =\eta _{2}$. We get the algebra ${\rm R}_{14}$.
\end{itemize}

\item 
\underline{${\rm R}^{+}={\rm J}_{19}$.} Let $\theta \ =\ \big(
B_{1},B_{2},B_{3}\big) \neq 0$ be an arbitrary element of  $ {\rm Z}^{2}\big( {\rm J}_{19}, {\rm J}_{19}\big) $. Then $\theta \in \big\{ \eta _{1},\ \eta
_{2}\big\} $ where%
\begin{longtable}{lcllcl}
$\eta _{1} $&$=$&$\big( 0,\ 0,\ \frac{1}{2}\Delta _{13}\big) ,$ &
$\eta _{2} $&$=$&$\big( 0,\ 0,\ -\frac{1}{2}\Delta _{13}\big) .$
\end{longtable}%
The automorphism group of ${\rm J}_{19}$, ${\rm Aut}\big( {\rm J}_{19}\big) $, consists of the invertible matrices of the following form:%
\begin{equation*}
\varphi =  
\begin{pmatrix} 
1 & 0 & 0 \\ 
0 & 1 & 0 \\ 
a_{31} & 0 & a_{33}%
\end{pmatrix}.
\end{equation*}%
The $\theta \ast \varphi =\theta $ for any $\varphi \in {\rm Aut}\big( 
{\rm J}_{19}\big) $. So we get the algebras ${\rm R}_{15}$ and $\rm{
R}_{16}$.
\end{enumerate}

\subsubsection{Corollaries from Theorem A1}
 
Obviously, all associative algebras are right alternative. So, as a byproduct, we have the following classification of complex $3$-dimensional associative algebras.

\begin{theorem}\label{3-dim assoc}
    Let ${\rm A}$\ be a complex $3$-dimensional associative algebra. Then $
{\rm A}$\ is isomorphic to one of commutative associative algebras listed in Theorem \ref{3-dim commassoc} or to one of the following algebras:
 
\begin{longtable}{llllllll}
 
${\rm A}_{12}$ & $:$ & $e_{1}e_{2} = e_{3}$ & $e_{2}e_{1} = -e_{3}$ & & & & \\ 
${\rm A}_{13}$ & $:$ & $e_{1}e_{1} = e_{2}$ & $e_{1}e_{3} = e_{2}$ & $e_{3}e_{1} = -e_{2}$ & & & \\ 
${\rm A}_{14}^{\alpha \neq 0}$ & $:$ & $e_{1}e_{2} = \big( 1+\alpha \big) e_{3}$ & $e_{2}e_{1} = \big( 1-\alpha \big) e_{3}$ & & & & \\ 
${\rm A}_{15}$ & $:$ & $e_{1}e_{1} = e_{1}$ & $e_{1}e_{2} = e_{2}$ & $e_{1}e_{3} = e_{3}$ & $e_{3}e_{1} = e_{3}$ & & \\ 
${\rm A}_{16}$ & $:$ & $e_{1}e_{1} = e_{1}$ & $e_{1}e_{3} = e_{3}$ & $e_{2}e_{1} = e_{2}$ & $e_{3}e_{1} = e_{3}$ & & \\ 
${\rm A}_{17}$ & $:$ & $e_{1}e_{1} = e_{1}$ & $e_{1}e_{2} = e_{2}$ & $e_{1}e_{3} = e_{3}$ & & & \\ 
${\rm A}_{18}$ & $:$ & $e_{1}e_{1} = e_{1}$ & $e_{2}e_{1} = e_{2}$ & $e_{3}e_{1} = e_{3}$ & & & \\ 
${\rm A}_{19}$ & $:$ & $e_{1}e_{1} = e_{1}$  & $e_{1}e_{3} = e_{3}$ & $e_{2}e_{1} = e_{2}$ & & & \\ 
${\rm A}_{20}$ & $:$ & $e_{1}e_{1} = e_{1}$ & $e_{1}e_{3} = e_{3}$& $e_{2}e_{2} = e_{2}$  & $e_{3}e_{2} = e_{3}$ & & \\ 
${\rm A}_{21}$ & $:$ & $e_{1}e_{1} = e_{1}$ & $e_{2}e_{1} = e_{2}$ & & & & \\ 
${\rm A}_{22}$ & $:$ & $e_{1}e_{1} = e_{1}$ & $e_{1}e_{2} = e_{2}$ & & & & \\ 
${\rm A}_{23}$ & $:$ & $e_{1}e_{1} = e_{1}$   & $e_{1}e_{3} = e_{3}$ & $e_{2}e_{2} = e_{2}$& & & \\ 
${\rm A}_{24}$ & $:$ & $e_{1}e_{1} = e_{1}$ & $e_{2}e_{2} = e_{2}$ & $e_{3}e_{1} = e_{3}$ & & & \\ 
\end{longtable}
 
\noindent
All listed algebras are non-isomorphic except: ${\rm A}_{14}^{\alpha
}\cong {\rm A}_{14}^{-\alpha }$.
\end{theorem}

\begin{definition}
An associative algebra is called a $\mathfrak{perm}$  algebra if
 the following identity holds:%
\begin{longtable}{lcl}
$abc $&$= $&$acb.$\end{longtable}
\end{definition}

\begin{corollary}\label{perm}
Let ${\rm A}$ be a complex $3$-dimensional $\mathfrak{perm}$ algebra. Then ${\rm A}$ is isomorphic to one of the algebras given in Theorem \ref{3-dim assoc}  except:  ${\rm A}_{15},$ ${\rm A}_{17},$ ${\rm A}_{19},$ ${\rm A}_{20},$ ${\rm A}_{22}$ and ${\rm A}_{23}$.
\end{corollary}

\begin{definition}
An algebra is called a binary $\mathfrak{perm}$ algebra
if each $2$-generated subalgebra is a $\mathfrak{perm}$ algebra\footnote{Kunanbayev
and  Sartayev proved that these algebras are   alternative and satisfy $(ab)c +(cb)a = (ac)b +(ca)b$ \cite{KS}.}.
\end{definition}
 
\begin{corollary}\label{bperm}
Let ${\rm A}$ be a complex $3$-dimensional binary $\mathfrak{perm}$ algebra. Then ${\rm A}$ is isomorphic to one of the algebras given in Corollary \ref{perm}, i.e. ${\rm A}$ is $\mathfrak{perm}$.
\end{corollary}

\begin{definition}
A right-alternative algebra is called a $\big(-1,1\big) $-algebra if
it is Lie-admissible, i.e. the following identity holds:%
\begin{longtable}{lcl}
$(x,y,z)+(y,z,x)+(z,x,y) $&$=$&$0.$
\end{longtable}
\end{definition}

\begin{corollary}\label{(-1,1)}
Let ${\rm A}$ be a complex $3$-dimensional $\big(-1,1\big) $-algebra. Then ${\rm A}$ is isomorphic to one of the  right-alternative algebras given in Theorem \ref{3-dim assoc}.
\end{corollary}

\begin{definition}
An algebra is called a binary $\big( -1,1\big) $-algebra
if each $2$-generated subalgebra is a  $\big( -1,1\big)$-algebra.
\end{definition}

For a $\big( -1,1\big)$-algebra ${\rm A}$ the commutator algebra ${\rm A}^{-}$ is a Lie algebra, and if ${\rm A}$
is binary $\big( -1,1\big)$-algebra then ${\rm A}^{-}$ is binary-Lie. Moreover, any binary $\big(-1,1\big) $-algebra is $\big(-1,1\big) $-algebra if and only if it is Lie-admissible. Since all complex $3$-dimensional binary-Lie algebras are Lie, all complex $3$-dimensional binary $\big( -1,1\big)$-algebras are Lie-admissible. So we have the following result:

\begin{corollary}\label{3dimbin-11}
Let ${\rm A}$ be a complex $3$-dimensional binary $\big(-1,1\big) $-algebra. Then ${\rm A}$ is isomorphic to one of the algebras given in Corollary \ref{(-1,1)}, i.e. ${\rm A}$ is  a $\left(-1,1\right) $-algebra.
\end{corollary}

Pchelintsev in \cite{ser76bi} proved that binary $\big(-1,1\big)$-algebras are right alternative and satisfy \begin{longtable}{lcl}
$( xy,x,y) +  ( y,xy,x) + ( x,y,xy) $&$=$&$0$.
\end{longtable}
 The linearization of the last  identity  is as follows:
 \begin{longtable}{lll}
$(wz,y,x)+(w,z,yx)+(wx,y,z)+(w,x,yz)+(x,yz,w)+(x,wz,y)$&$+$\\
$(yz,w,x)+(y,z,wx)+(yx,w,z)+(y,x,wz)+(z,yx,w)+(z,wx,y)$&$=$&$0.$
\end{longtable}
Equivalently:
\begin{longtable}{lc}
$w(z(yx))+w(x(yz))+x((yz)w)+x((wz)y)+y(z(wx))+y(x(wz))+z((yx)w)+z((wx)y)$&$=$\\
\multicolumn{2}{r}{$(z(yx))w+(x(yz))w+((yz)w)x+((wz)y)x+(z(wx))y+(x(wz))y+((yx)w)z+((wx)y)z.$}
\end{longtable}

 \begin{lemma} \label{4-dim Bl}\footnote{The classification of $4$-dimensional binary Lie algebras is given by Gainov in 1963 \cite{ikp20}.}
     Let ${\rm B}$ be a $4$-dimensional complex non-Lie binary Lie algebra.
Then ${\rm B}$ is isomorphic to one (and only one) of the following
algebras:

\begin{longtable}{lcllll}
${\rm B}_{1}^{\alpha \neq 2}$&$:$ & $\left[ e_{1},e_{2}\right] =e_{2}$ & 
$\left[ e_{1},e_{3}\right] =e_{3}$ &  $\left[ e_{1},e_{4}\right] =\alpha e_{4}$ & $\left[ e_{2},e_{3}\right] =e_{4}$\\

${\rm B}_{2}$&$:$& $\left[ e_{1},e_{2}\right] =e_{3}$ & $\left[e_{3},e_{4}\right] =e_{3}$
\end{longtable}
 \end{lemma}

 \begin{theorem} \label{non (-1,1)}
     Let ${\mathbb B}$ be a complex $4$-dimensional binary $\big(-1,1\big) $-algebra. 
     Then ${\mathbb B}$\ is a $\big(-1,1\big) $-algebra or it is isomorphic
to one of the following algebras\footnote{For receiving similar multiplication tables we have to apply the basis change $e_1:=-e_1$ in algebras ${\mathbb B}_{04},$\ ${\mathbb B}_{05}^{\alpha},$\ ${\mathbb B}_{06},$ and ${\mathbb B}_{08}$.}:
 
\begin{longtable}{llllll}
${\mathbb B}_{01}$ & $:$ & $e_{1} e_{1} = 2e_{1}$ & $e_{1} e_{2} = 2e_{2}$ & $e_{1} e_{3} = 2e_{3}$ & $e_{1} e_{4} = 2e_{4}$ \\
 & & $e_{4} e_{1} = 2e_{4}$ & $e_{2} e_{3} = e_{4}$ & $e_{3} e_{2} = -e_{4}$ & $e_{3} e_{3} = -e_{4}$ \\
${\mathbb B}_{02}^{\alpha}$ & $:$ & $e_{1} e_{1} = 2e_{1}$ & $e_{1} e_{2} = 2e_{2}$ & $e_{1} e_{3} = 2e_{3}$ & $e_{1} e_{4} = 2e_{4}$ \\
 & & $e_{4} e_{1} = 2e_{4}$ & $e_{2} e_{3} = (\alpha + 1)e_{4}$ & $e_{3} e_{2} = (\alpha - 1)e_{4}$ & \\
${\mathbb B}_{03}$ & $:$ & $e_{1} e_{1} = 2e_{1}$ & $e_{1} e_{2} = 2e_{2} + e_{4}$ & $e_{2} e_{1} = e_{4}$ & $e_{1} e_{3} = 2e_{3}$ \\
 & & $e_{1} e_{4} = 2e_{4}$ & $e_{4} e_{1} = 2e_{4}$ & $e_{2} e_{3} = 2e_{4}$ & \\

{${\mathbb B}_{04}$} & $:$ & $e_{1} e_{1} = 2e_{1}$ & $e_{2} e_{1} = 2e_{2}$ & $e_{3} e_{1} = 2e_{3}$ & $e_{2} e_{3} = e_{4}$ \\
 & & $e_{3} e_{2} = -e_{4}$ & $e_{3} e_{3} = -e_{4}$ & & \\


{${\mathbb B}_{05}^{\alpha}$} & $:$ & $e_{1} e_{1} = 2e_{1}$ & $e_{2} e_{1} = 2e_{2}$ & $e_{3} e_{1} = 2e_{3}$ & $e_{2} e_{3} = (\alpha + 1)e_{4}$ \\
 & & $e_{3} e_{2} = (\alpha - 1)e_{4}$ & & & \\
 

{${\mathbb B}_{06}$} & $:$ & $e_{1} e_{1} = 2e_{1}$ & $e_{1} e_{2} = -e_{4}$ & $e_{2} e_{1} = 2e_{2} - e_{4}$ & $e_{3} e_{1} = 2e_{3}$ \\
 & & $e_{2} e_{3} = 2e_{4}$ & & & \\
 
${\mathbb B}_{07}$ & $:$ & $e_{1} e_{1} = 2e_{1}$ & $e_{1} e_{2} = 2e_{2}$ & $e_{1} e_{3} = 2e_{3}$ & $e_{4} e_{1} = 2e_{4}$ \\
 & & $e_{2} e_{3} = e_{4}$ & $e_{3} e_{2} = -e_{4}$ & & \\


{${\mathbb B}_{08}$} & $:$ & $e_{1} e_{1} = 2e_{1}$ & $e_{2} e_{1} = 2e_{2}$ & $e_{3} e_{1} = 2e_{3}$ & $e_{1} e_{4} = 2e_{4}$ \\
 & & $e_{2} e_{3} = e_{4}$ & $e_{3} e_{2} = -e_{4}$ & & \\
\end{longtable}
 
\noindent All listed algebras are non-isomorphic except: ${\mathbb B}_{02}^{\alpha}\cong {\mathbb B}_{02}^{-\alpha }$ and ${\mathbb B}_{05}^{\alpha}\cong {\mathbb B}_{05}^{-\alpha}$.
 \end{theorem}

\begin{proof}
    Let ${\mathbb B}$ be a complex $4$-dimensional binary $\big(-1,1\big) $-algebra. Suppose that ${\mathbb B}$ is non $\big(-1,1\big) $-algebra. Then ${\mathbb B}^{-}$ is a non-Lie binary Lie algebra. By Lemma \ref{4-dim Bl}, we may assume ${\mathbb B}
^{-}\in \big\{{\rm B}_{1}^{\alpha \neq 2}, \ {\rm B}_{2}\big\}$. If ${\mathbb B}
^{-}\in \big\{{\rm B}_{1}^{\alpha \neq -1,0},\ {\rm B}_{2}\big\}$, 
then ${\rm Z}^{2}\big( {\mathbb B}^{-},{\mathbb B}^{-}\big) =\varnothing $\footnote{${\rm Z}^2$ is defined in a similar way to definition \ref{Z2}, but we have to use the identities of binary $\big( -1,1\big)$-algebras instead of the identities of semi-alternative algebras.}
\footnote{ In \cite{AS}  after the proof of Proposition 2, the following was mentioned:
{\it It remains an open question whether any binary Lie algebra is (binary $(-1,1))$-special, that is, can be embedded into a commutator algebra ${\rm A}^-$ for a certain binary $(-1,1)$-algebra ${\rm A}$}. 
As shown in Theorem 20, ${\rm Z}^2({\rm B}_2,{\rm B}_2) = \varnothing$. 
This gives a particular answer to the mentioned question: 
there are binary Lie algebras that can not be embedded in binary $(-1,1)$-algebras of the same dimension under the commutator product.}. So we study the
following cases:

\begin{enumerate}[I.]

\item \underline{${\mathbb B}^{-}={\rm B}_{1}^{0}.$} 
Let $\theta \ =\ \big(
B_{1},B_{2},B_{3}, B_4\big) \neq 0$ be an arbitrary element of  $ {\rm Z}^{2}\big( 
{\rm B}_{1}^{0},{\rm B}_{1}^{0}\big) $. Then $\theta \in \left\{ \eta _{1},\eta
_{2}\right\},$ where%
\begin{longtable}{lcl}
$\eta _{1} $&$=$&$\big( 2\Delta _{11},\ \Delta _{12},\ 
\Delta_{13},\ 
\alpha _{1}\Delta _{11}+\alpha _{2}\Delta _{12}+\alpha
_{3}\Delta _{13}+2\Delta _{14}+\alpha _{4}\Delta _{22}+\alpha
_{5}\Delta _{23}+\alpha _{6}\Delta _{33}\big),$  \\
$\eta _{2} $&$=$&$\big( -2\Delta _{11},\ 
-\Delta _{12},\ 
-\Delta_{13},\ 
\alpha _{1}\Delta _{11}+\alpha _{2}\Delta _{12}+\alpha
_{3}\Delta _{13}+\alpha _{4}\Delta _{22}+\alpha _{5}\Delta
_{23}+\alpha _{6}\Delta _{33}\big) ,$
\end{longtable}%
for some $\alpha _{1},\alpha _{2},\alpha _{3},\alpha _{4},\alpha _{5},\alpha
_{6}\in \mathbb{C}$. The automorphism group of ${\rm B}_{1}^{0}$, $\mathrm{Aut}%
\big( {\rm B}_{1}^{0}\big) $, consists of the invertible matrices of the
following form:%
\begin{equation*}
\varphi =%
\begin{pmatrix}
1 & 0 & 0 & 0 \\ 
a_{21} & a_{22} & a_{23} & 0 \\ 
a_{31} & a_{32} & a_{33} & 0 \\ 
a_{41} & a_{21}a_{32}-a_{22}a_{31} & a_{21}a_{33}-a_{31}a_{23} & 
a_{22}a_{33}-a_{23}a_{32}%
\end{pmatrix}%
.
\end{equation*}

\begin{itemize}
\item $\theta =\eta _{1}$. Let $\varphi =\bigl(a_{ij}\bigr)\in $ $\rm{Aut}%
\left({\rm B}_{1}^{0}\right) $. Then%
\begin{longtable}{lcl}
$\theta \ast \varphi$&$ =$&$\big( 2\Delta _{11},\ 
\Delta _{12},\ 
\Delta_{13},\ 
\beta _{1}\Delta _{11}+\beta _{2}\Delta _{12}+\beta
_{3}\Delta _{13}+2\Delta _{14}+\beta _{4}\Delta _{22}+\beta
_{5}\Delta _{23}+\beta _{6}\Delta _{33}\big),$
\end{longtable}%
where%
\begin{longtable}{lcl}
$\beta _{4} $&$=$&$\frac{\alpha
_{4}a_{22}^{2}+2\alpha _{5}a_{22}a_{32}+\alpha _{6}a_{32}^{2}}{a_{22}a_{33}-a_{23}a_{32}} ,$ \\
$\beta _{5} $&$=$&$\frac{\alpha
_{4}a_{22}a_{23}+\alpha _{5}a_{22}a_{33}+\alpha _{5}a_{23}a_{32}+\alpha
_{6}a_{32}a_{33}}{a_{22}a_{33}-a_{23}a_{32}} ,$ \\
$\beta _{6} $&$=$&$\frac{\alpha
_{4}a_{23}^{2}+2\alpha _{5}a_{23}a_{33}+\alpha _{6}a_{33}^{2}}{a_{22}a_{33}-a_{23}a_{32}}.$
\end{longtable}%
Hence $%
\begin{pmatrix}
-\beta _{5} & -\beta _{6} \\ 
\beta _{4} & \beta _{5}%
\end{pmatrix}%
=%
\begin{pmatrix}
a_{22} & a_{23} \\ 
a_{32} & a_{33}%
\end{pmatrix}%
^{-1}%
\begin{pmatrix}
-\alpha _{5} & -\alpha _{6} \\ 
\alpha _{4} & \alpha _{5}%
\end{pmatrix}%
\begin{pmatrix}
a_{22} & a_{23} \\ 
a_{32} & a_{33}%
\end{pmatrix}%
$. $\allowbreak $So we may assume $%
\begin{pmatrix}
-\alpha _{5} & -\alpha _{6} \\ 
\alpha _{4} & \alpha _{5}%
\end{pmatrix}%
\in \left\{ 
\begin{pmatrix}
0 & 1 \\ 
0 & 0%
\end{pmatrix}%
,%
\begin{pmatrix}
-\alpha  & 0 \\ 
0 & \alpha 
\end{pmatrix}%
\right\} $.

\begin{itemize}
\item $%
\begin{pmatrix}
-\alpha _{5} & -\alpha _{6} \\ 
\alpha _{4} & \alpha _{5}%
\end{pmatrix}%
=%
\begin{pmatrix}
0 & 1 \\ 
0 & 0%
\end{pmatrix}%
$. Consider the following automorphism:%
\begin{equation*}
\varphi _{1}=%
\begin{pmatrix}
1 & 0 & 0 & 0 \\ 
\alpha _{2}-\alpha _{3} & 1 & 0 & 0 \\ 
\alpha _{2} & 0 & 1 & 0 \\ 
-\frac{1}{2}\alpha _{2}^{2}-\frac{1}{2}\alpha _{1} & -\alpha _{2} & \alpha
_{2}-\alpha _{3} & 1%
\end{pmatrix}%
\allowbreak .
\end{equation*}%
Then $\theta \ast \varphi _{1}=\big( 2\Delta _{11},\ 
\Delta_{12},\ \Delta _{13},\ 2\Delta _{14}-\Delta _{33}\big) $. So we
get the algebra ${\mathbb B}_{01}$.

\item $%
\begin{pmatrix}
-\alpha _{5} & -\alpha _{6} \\ 
\alpha _{4} & \alpha _{5}%
\end{pmatrix}%
=%
\begin{pmatrix}
-\alpha  & 0 \\ 
0 & \alpha 
\end{pmatrix}%
$.

\begin{itemize}
\item $\alpha ^{2}\neq 1$. Consider the following automorphism:%
\begin{equation*}
\varphi _{2}=%
\begin{pmatrix}
1 & 0 & 0 & 0 \\ 
-\frac{\alpha _{3}}{\alpha +1} & 1 & 0 & 0 \\ 
-\frac{\alpha _{2}}{\alpha -1} & 0 & 1 & 0 \\ 
\frac{\alpha _{1}-\alpha ^{2}\alpha _{1}+2\alpha \alpha _{2}\alpha _{3}}{%
2\alpha ^{2}-2} & \frac{\alpha _{2}}{\alpha -1} & -\frac{\alpha _{3}}{\alpha
+1} & 1%
\end{pmatrix}%
.\allowbreak 
\end{equation*}%
Then $\theta \ast \varphi _{2}=\big( 2\Delta _{11},\ 
\Delta_{12}, \ \Delta _{13}, \ 2\Delta _{14}+\alpha \Delta _{23}\big) $.
So we get the algebras ${\mathbb B}_{02}^{\alpha \neq \pm 1}$.

\item $\alpha =1$. Consider the following matrices and let $\varphi =\varphi _{3}$
if $\alpha _{2}=0$ or $\varphi =\varphi _{4}$ if $\alpha _{2}\neq 0$:%
\begin{equation*}
\varphi _{3}=%
\begin{pmatrix}
1 & 0 & 0 & 0 \\ 
-\frac{1}{2}\alpha _{3} & 1 & 0 & 0 \\ 
0 & 0 & 1 & 0 \\ 
-\frac{1}{2}\alpha _{1} & 0 & -\frac{1}{2}%
\alpha _{3} & 1%
\end{pmatrix}%
, \ \varphi _{4}=%
\begin{pmatrix}
1 & 0 & 0 & 0 \\ 
-\frac{1}{2}\alpha _{3} & 1 & 0 & 0 \\ 
0 & 0 & \alpha _{2} & 0 \\ 
\frac{1}{2}\alpha _{2}\alpha _{3}-\frac{1}{2}\alpha _{1} & 0 & -\frac{1}{2}%
\alpha _{2}\alpha _{3} & \alpha _{2}%
\end{pmatrix}%
.
\end{equation*}%
Then 
\begin{center}
    $\theta \ast \varphi \in \big\{\big( 2\Delta _{11},\ 
    \Delta_{12},\ \Delta _{13},\  2\Delta _{14}+\Delta _{23}\big),\ 
    \big(
2\Delta _{11},\ \Delta _{12},\ \Delta _{13},\ \Delta_{12}+2\Delta _{14}+\Delta _{23}\big) \big\}$. 
\end{center} So we get the algebras 
${\mathbb B}_{02}^{1}$ and ${\mathbb B}_{03}$.

\item $\alpha =-1$. Consider the following matrices and let $\varphi =\varphi _{5}$
if $\alpha _{3}=0$ or $\varphi =\varphi _{6}$ if $\alpha _{3}\neq 0$:%
\begin{equation*}
\varphi _{5}=%
\begin{pmatrix}
1 & 0 & 0 & 0 \\ 
0 & 0 & -1 & 0 \\ 
\frac{1}{2}\alpha _{2} & 1 & 0 & 0 \\ 
-\frac{1}{2}\alpha _{1} & 0 & \frac{1}{2}%
\alpha _{2} & 1%
\end{pmatrix},
\ \varphi _{6}=%
\begin{pmatrix}
1 & 0 & 0 & 0 \\ 
0 & 0 & -\alpha _{3} & 0 \\ 
\frac{1}{2}\alpha _{2} & 1 & 0 & 0 \\ 
-\frac{1}{2}\alpha _{1}-\frac{1}{2}\alpha _{2}\alpha _{3} & 0 & \frac{1}{2}%
\alpha _{2}\alpha _{3} & \alpha _{3}%
\end{pmatrix}%
.
\end{equation*}%
Then 
\begin{center}
    $\theta \ast \varphi \in \big\{\big( 2\Delta _{11},\ 
    \Delta_{12},\ \Delta _{13},\ 2\Delta _{14}+\Delta _{23}\big), \ 
    \big(2\Delta _{11}, \ \Delta _{12},\ \Delta _{13},\ \Delta
_{12}+2\Delta _{14}+\Delta _{23}\big) \big\}$.
\end{center} So we get the algebras 
${\mathbb B}_{02}^{1}$ and $ {\mathbb B}_{03}$.
\end{itemize}
\end{itemize}

\item $\theta =\eta _{2}$. As in the case $\theta =\eta _{1}$, we may assume 
$%
\begin{pmatrix}
-\alpha _{5} & -\alpha _{6} \\ 
\alpha _{4} & \alpha _{5}%
\end{pmatrix}%
\in \left\{ 
\begin{pmatrix}
0 & 1 \\ 
0 & 0%
\end{pmatrix}%
,%
\begin{pmatrix}
-\alpha  & 0 \\ 
0 & \alpha 
\end{pmatrix}%
\right\} $.

\begin{itemize}
\item $%
\begin{pmatrix}
-\alpha _{5} & -\alpha _{6} \\ 
\alpha _{4} & \alpha _{5}%
\end{pmatrix}
=%
\begin{pmatrix}
0 & 1 \\ 
0 & 0%
\end{pmatrix}%
$. Then $\theta \ast \varphi _{1}=\big( -2\Delta _{11}, \ -\Delta
_{12},\ -\Delta _{13},\ -\Delta _{33}\big) $. So we get the algebra $%
{\mathbb B}_{04}$.

\item $%
\begin{pmatrix}
-\alpha _{5} & -\alpha _{6} \\ 
\alpha _{4} & \alpha _{5}%
\end{pmatrix}
=%
\begin{pmatrix}
-\alpha  & 0 \\ 
0 & \alpha 
\end{pmatrix}%
$.

\begin{itemize}
\item $\alpha ^{2}\neq 1$. Then $\theta \ast \varphi _{2}=\big( -2\Delta
_{11},\ -\Delta _{12},\ -\Delta _{13},\ \alpha \Delta _{23}\big) $.
So we get the algebras ${\mathbb B}_{05}^{\alpha \neq \pm 1}$.

\item $\alpha =1$. Let $\varphi =\varphi _{3}$ if $\alpha _{2}=0$ or $\varphi =\varphi
_{4}$ if $\alpha _{2}\neq 0$. Then 
\begin{center}
    $\theta \ast \varphi \in \big\{\big(
-2\Delta _{11},\ -\Delta _{12},\ -\Delta _{13},\ 
\Delta
_{23}\big) ,\ \big( -2\Delta _{11},\ -\Delta _{12},\ 
-\Delta_{13},\ \Delta _{12}+\Delta _{23}\big) \big\}$.
\end{center} So we get the algebras $%
{\mathbb B}_{05}^{1}$ and ${\mathbb B}_{06}$.

\item $\alpha =-1$. Let $\varphi =\varphi _{5}$ if $\alpha _{3}=0$ or $\varphi =\varphi
_{6}$ if $\alpha _{3}\neq 0$. Then 
\begin{center}$\theta \ast \varphi \in \big\{\big(
-2\Delta _{11},\ -\Delta _{12},\ -\Delta _{13},\ \Delta
_{23}\big) , \ \big( -2\Delta _{11},\ -\Delta _{12},\ -\Delta
_{13},\ \Delta _{12}+\Delta _{23}\big) \big\}$.
\end{center} So we get the algebras $%
{\mathbb B}_{05}^{1}$ and ${\mathbb B}_{06}$.
\end{itemize}
\end{itemize}
\end{itemize}

\item 
\underline{${\rm A}^{-}={\rm B}_{1}^{-1}$.} 
Let $\theta \ =\ \big(
B_{1},B_{2},B_{3}, B_4\big) \neq 0$ be an arbitrary element of  ${\rm Z}^{2}\big( {\rm B}_{1}^{-1},{\rm B}_{1}^{-1}\big) $. Then $\theta \in \big\{
\eta _{1},\ \eta _{2}\big\},$ where%
\begin{longtable}{lcllcl}
$\eta _{1} $&$=$&$\big( 2\Delta _{11},\ \Delta _{12},\ \Delta
_{13},\ \Delta _{14}\big) ,$ &
$\eta _{2} $&$=$&$\big( -2\Delta _{11},\ -\Delta _{12},\ -\Delta_{13},\ -\Delta _{14}\big).$
\end{longtable}
Thus we get the algebra ${\mathbb B}_{07}$ if $\theta =\eta _{1}$ and the
algebra ${\mathbb B}_{08}$ if $\theta =\eta _{2}$.
\end{enumerate}
\end{proof}

 \begin{corollary}\label{21}
The varieties of 
$\mathfrak{perm}$ and binary $\mathfrak{perm}$  algebras are coinciding 
in dimensions $\leq 3,$ but they are distinct in dimension $4$.
The varieties of 
$\big(-1,1\big)$-, binary $\big(-1,1\big)$- and right alternative algebras are coinciding 
in dimensions $\leq 3,$ but they are distinct in dimension $4$.
\end{corollary}

 \begin{proof}
By Corollary \ref{bperm}, it follows that all complex $3$-dimensional binary $\mathfrak{perm}$ algebras are $\mathfrak{perm}$ and thus associative. The situation is different in dimension $4$ since the following algebra is an example of a nonassociative binary $\mathfrak{perm}$ algebra which is not $\mathfrak{perm}$ because it is nonassociative.
\begin{longtable}{lclllllll}
${\rm P}$& $:$& $e_{1}e_{2}=2e_{4}$ & $e_{1}e_{3}=e_{2}$ & $e_{3}e_{1}=-e_{2}$ & $e_{3}e_{3}=e_{4}$
\end{longtable}

 The algebras given in Theorem \ref{non (-1,1)} are   right alternative but they are not $\big(-1,1\big)$-algebras.   
 The algebra given below is right alternative, but it is not binary $\big(-1,1\big)$-.

\begin{longtable}{lclllllll}
$R$&$:$&
$e_{1}e_{2}=e_{3}$&$e_{2}e_{1}=-e_{3}$&$e_{3}e_{4}=2e_{3}$&$e_{1}e_{4}=2e_{1}$\\ &&
$e_{4}e_{1}=2e_{1}$&$e_{2}e_{4}=2e_{2}$&$e_{4}e_{2}=2e_{2}$&$e_{4}e_{4}=2e_{4}$
\end{longtable}

 \end{proof}

Let us remember that each $3$-dimensional binary associative (=alternative) algebra is associative \cite{G62};
each $3$-dimensional binary Lie algebras is Lie   \cite{ikp20};
each $3$-dimensional binary $\mathfrak{perm}$ algebra is a $\mathfrak{perm}$ algebra (Corollary \ref{bperm}); 
 each $3$-dimensional binary $\big( -1,1\big)$-algebra is a  $\big( -1,1\big)$-algebra (Corollary \ref{3dimbin-11}).   
 The present observation raised the following question.
\medskip 

\noindent{\bf Open question.}
Is it true that each $3$-dimensional binary-$\Omega$\footnote{
By a binary $\Omega$ algebra we mean an algebra such that each $2$-generated subalgebra is an $\Omega$ algebra for a family of polynomial identities $\Omega$.}
 algebra is an $\Omega$ algebra?

 \begin{remark}
     Let us note that a similar question for $2$-dimensional mono $\Omega$-algebras does not make  sense. 
     Namely, the variety of mono-associative commutative algebras (i.e. each one-generated algebra is commutative associative) gives the variety of power-associative algebras. Still, the nontrivial $2$-dimensional anticommutative algebra (with the multiplication table $e_1e_2=e_2,$ \ $e_2e_1=-e_2$) is power-associative, but not commutative associative.

 \end{remark}

\subsection{Semi-alternative algebras}

\begin{definition}
An algebra ${\rm S}$ is called  semi-alternative if the following identity holds
\begin{center}
    $\big(
x,y,z\big) =\big( y,z,x\big).$
\end{center}
\end{definition}

The next Lemma follows from  \cite[identity (5)]{AS}  and \cite[identity (10)]{V93}.

\begin{lemma}
Let $S$ be a semi-alternative algebra. Then $S$ satisfies the following
identites:%
\begin{equation}
J\big( a,b,c\big) \ =\ 3\big( a,b,c\big) -3\big(a,c,b\big) ,
\label{id1-semi}
\end{equation}%
\begin{equation}
\begin{array}{lllllllllll}
3\big( \left[ a,b\right] ,c,d\big) & =& 
-&\left[ a,\big( b,c,d\big) \right] &+&\left[ b,\big( a,c,d\big) \right] &+&\left[ c,\big( a,b,d\big) \right] \\
&&-&
\left[ c,\big( b,a,d\big) \right] &-&\left[ d,\big( a,b,c\big) \right] &+&
\left[ d,\big( b,a,c\big) \right],
\end{array}\label{id2-semi}
\end{equation}%
where the jacobian $J\big( a,b,c\big) $ is taken with respect to the
commutator product $\left[\cdot,\cdot\right] $.
\end{lemma}
In \cite[Theorem 1]{AS}, it is proved that the commutator algebra of any semi-alternative algebra is a binary Lie algebra and in fact the following stronger result holds.
\begin{lemma}\label{commutator}
Let ${\rm S}$ be a semi-alternative algebra. Then $({\rm S}, [\cdot,\cdot])$ is a Malcev algebra\footnote{
The present observation gives an easy answer 
to problem 2.109 from the Dniester Notebook [2.109: An algebra is called assocyclic if it satisfies the identity
$(x, y, z) = (z, x, y)$ where $(x, y, z) = (xy)z - x(yz).$ It is easy to show that the
minus algebra of such an algebra is binary-Lie. Is every binary-Lie algebra over
a field of characteristic $\neq 2, 3$ embeddable into the minus algebra of a suitable
assocyclic algebra?], 
previously resolved by Shestakov and Arenas \cite{AS}.}
.
If ${\rm S}$ is anticommutative, then ${\rm S}$ is $3$-step nilpotent.
If ${\rm S}$ is commutative, then ${\rm S}$ is associative.
\end{lemma}

\begin{proof}
Let $x,y,z\in {\rm S}$. Then, from (\ref{id2-semi}), we have
\begin{longtable}{lclclclclcl}
$3\big( \left[ x,z\right] ,y,x\big)  $&$=$&
$-\left[ x,\big( z,y,x\big) \right] $&$+$&$\left[ z,\big( x,y,x\big) \right] $&$+$&$\left[ y,\big( x,z,x\big) %
\right] $\\ &&&$-$&$\left[ y,\big( z,x,x\big) \right] $&$-$&$\left[ x,\big( x,z,y\big) %
\right] $&$+$&$\left[ x,\big( z,x,y\big) \right] ,$ \\
$3\big( \left[ x,z\right] ,x,y\big)  $&$=$&$-\left[ x,\big( z,x,y\big) %
\right] $&$+$&$\left[ z,\big( x,x,y\big) \right] $&$+$&$\left[ x,\big( x,z,y\big) %
\right] $\\ &&&$-$&$\left[ x,\big( z,x,y\big) \right] $&$-$&$\left[ y,\big( x,z,x\big) %
\right] $&$+$&$\left[ y,\big( z,x,x\big) \right] .$
\end{longtable}%
Since ${\rm S}$  is semi-alternative, we get%
\begin{center}
$3\big( \left[ x,z\right] ,y,x\big) -3\big( \left[ x,z\right] ,x,y\big)
\ =\ \left[ x,3\big( x,y,z\big) -3\big( x,z,y\big) \right].$
\end{center}
Then, from (\ref{id1-semi}), we have 
\begin{center}
$J\big( \left[ x,z\right] ,y,x\big) =\left[ x,J\big( x,y,z\big) \right],$
i.e.  $({\rm S}, [\cdot,\cdot])$ is a Malcev algebra.
\end{center}

The second statement follows from the observation that each anticommutative semi-alternative algebra is 
antiassociative (i.e. $(xy)z+x(yz)=0$), hence $3$-step nilpotent.

The third   statement follows from the observation that each commutative semi-alternative algebra is 
$\frac{1}{2}$-Leibniz (i.e. $x(yz)=\frac{1}{2}\big( (xy)z+y(xz)\big)$), hence associative.

\end{proof}

\begin{theorem}
\label{3-dim lie}Let $\big( \mathcal{L},\left[\cdot,\cdot\right] \big) $ be a
nontrivial complex $3$-dimensional Malcev algebra. Then $\mathcal{L}$ is
isomorphic to one of the following Lie algebras:

\begin{longtable}{lclll}
$\mathcal{L}_{01}$ & $:$ & $[e_{1}, e_{2}] = e_{3}$ \\ 
$\mathcal{L}_{02}$ & $:$ & $[e_{1}, e_{2}] = e_{2}$ & $[e_{1}, e_{3}] = e_{2} + e_{3}$ \\ 
$\mathcal{L}_{03}^{\alpha}$ & $:$ & $[e_{1}, e_{2}] = e_{2}$ & $[e_{1}, e_{3}] = \alpha e_{3}$ \\ 
$\mathcal{L}_{04}$ & $:$ & $[e_{1}, e_{2}] = e_{3}$ & $[e_{1}, e_{3}] = -2e_{1}$ & $[e_{2}, e_{3}] = 2e_{2}$ 
\end{longtable}
\end{theorem}

\subsubsection{The classification Theorem A2}

\begin{theoremA2}
Let ${\rm S}$ be a complex $3$-dimensional semi-alternative algebra. Then $%
{\rm S}$\ is isomorphic to one of associative algebras listed in Theorem \ref{3-dim assoc} or to one of the following algebras\footnote{For receiving similar multiplication tables we have to apply the basis change $e_1:=-e_1$ in algebras ${\rm S}_{05},$\ ${\rm S}_{06},$\ ${\rm S}_{07},$ and ${\rm S}_{13}$.}:
\begin{longtable}{llllllllll}
 ${\rm S}_{01}^{\alpha }$ & $:$ & $e_{1} e_{1} = e_{2}$ & $e_{1} e_{2} = \big( \alpha +1\big) e_{3}$ & $e_{2} e_{1} = \big( \alpha -1\big) e_{3}$ & & & \\ 
${\rm S}_{02}$ & $:$ & $e_{1} e_{1} = 2e_{1} + e_{2}$ & $e_{1} e_{2} = 2e_{2}$ & $e_{3} e_{1} = 2e_{3}$ & & & \\ 
${\rm S}_{03}$ & $:$ & $e_{1} e_{1} = 2e_{1} + e_{2} + e_{3}$ & $e_{1} e_{2} = 2e_{2}$ & $e_{3} e_{1} = 2e_{3}$ & & & \\ 
${\rm S}_{04}$ & $:$ & $e_{1} e_{1} = 2e_{1} + e_{3}$ & $e_{1} e_{2} = 2e_{2}$ & $e_{3} e_{1} = 2e_{3}$ & & & \\ 

{${\rm S}_{05}$} & $:$ & $e_{1} e_{1} = 2e_{1} + e_{2}$ & $e_{2} e_{1} = 2e_{2}$ & $e_{1} e_{3} = 2e_{3}$ & $e_{3} e_{1} = 2e_{3}$ & & \\ 


{${\rm S}_{06}$} & $:$ & $e_{1} e_{1} = 2e_{1} + e_{2}$ & $e_{2} e_{1} = 2e_{2}$ & $e_{3} e_{3} = e_{3}$ & & & \\ 


{${\rm S}_{07}$} & $:$ & $e_{1} e_{1} = 2e_{1} + e_{2}$ & $e_{2} e_{1} = 2e_{2}$ & & & & \\

${\rm S}_{08}$ & $:$ & $e_{1} e_{1} = 2e_{1} + e_{2}$ & $e_{1} e_{2} = 2e_{2}$ & $e_{1} e_{3} = 2e_{3}$ & $e_{3} e_{1} = 2e_{3}$ & & \\ 
${\rm S}_{09}$ & $:$ & $e_{1} e_{1} = 2e_{1} + e_{2}$ & $e_{1} e_{2} = 2e_{2}$ & $e_{3} e_{3} = e_{3}$ & & & \\ 
${\rm S}_{10}$ & $:$ & $e_{1} e_{1} = 2e_{1} + e_{2}$ & $e_{1} e_{2} = 2e_{2}$ & & & & \\ 
${\rm S}_{11}$ & $:$ & $e_{1} e_{1} = e_{2} + e_{3}$ & $e_{1} e_{2} = e_{2}$ & $e_{2} e_{1} = -e_{2}$ & $e_{1} e_{3} = e_{1}$ & \\ 
& & $e_{3} e_{1} = e_{1}$ &  $e_{2} e_{3} = e_{2}$ & $e_{3} e_{2} = e_{2}$ & $e_{3} e_{3} = e_{3}$ & & & \\ 
${\rm S}_{12}$ & $:$ & $e_{1} e_{1} = 2e_{1} + e_{2}$ & $e_{1} e_{2} = 2e_{2}$ & $e_{1} e_{3} = 2e_{3}$ & & & \\ 
{${\rm S}_{13}$} & $:$ & $e_{1} e_{1} = 2e_{1} + e_{2}$ & $e_{2} e_{1} = 2e_{2}$ & $e_{3} e_{1} = 2e_{3}$ & & & \\

\end{longtable}
\end{theoremA2}
\subsubsection{Proof of Theorem A2}    

Let ${\rm S}$\ be a complex $3$-dimensional semi-alternative algebra.
Suppose that ${\rm S}$ is nonassociative. By Lemma \ref{commutator}
and Theorem \ref{3-dim lie}, we may assume ${\rm S}^{-}\in \{\mathcal{L}%
_{01},\mathcal{L}_{02},\mathcal{L}_{03}^{\alpha },\mathcal{L}_{04}\}$. Then $%
{\rm Z}^{2}\big( {\rm S}^{-},{\rm S}^{-}\big) \neq \varnothing $ if
and only if ${\rm S}^{-}\in \{\mathcal{L}_{01},\mathcal{L}_{03}^{-1},\mathcal{L}_{03}^{0},\mathcal{L}_{03}^{1}\}$. So we
study the following cases:

\begin{enumerate}[I.]
    \item 
\underline{${\rm S}^{-}=\mathcal{L}_{01}$.} Let $\theta \ =\ \big(
B_{1},B_{2},B_{3}\big) \neq 0$ be an arbitrary element of  $%
{\rm Z}^{2}\big( \mathcal{L}%
_{01},\mathcal{L}_{01}\big) $. Then $\theta \in \big\{ \eta _{1}, \ \ldots
, \ \eta _{4}\big\} $ where%
\begin{longtable}{lcll}
$\eta _{1} $&$=$&$\big( 0,\ 0,\ \alpha _{1}\Delta _{11}+\alpha
_{2}\Delta _{22}+\alpha _{3}\Delta _{12}\big) ,$ \\
$\eta _{2} $&$=$&$\big( 0,\ \alpha _{1}\Delta _{11}, \ \alpha
_{2}\Delta _{11}+\alpha _{3}\Delta _{12}\big) ,$ \\
$\eta _{3} $&$=$&$\big( \alpha _{1}\Delta _{22}, \ 0,\ \alpha
_{2}\Delta _{22}+\alpha _{3}\Delta _{12}\big) ,$ \\
$\eta _{4} $&$=$&$\big( -\alpha _{2}\Delta _{11}-\frac{\alpha _{1}^{2}}{
\alpha _{2}}\Delta _{22}+\alpha _{1}\Delta _{12},\ -\frac{%
\alpha _{2}^{2}}{\alpha _{1}}\Delta _{11}-\alpha _{1}\Delta
_{22}+\alpha _{2}\Delta _{12},$\\
\multicolumn{3}{r}{$\alpha _{4}\Delta _{11}-\frac{%
\alpha _{1}^{2}\alpha _{4}+2\alpha _{1}\alpha _{2}\alpha _{3}}{\alpha
_{2}^{2}}\Delta _{22}+\alpha _{3}\Delta _{12}$}&$\big)_{\alpha
_{1}\alpha _{2}\neq 0},$ 
\end{longtable}%
for some $\alpha _{1},\alpha _{2},\alpha _{3},\alpha _{4}\in \mathbb{C}$.
The automorphism group of $\mathcal{L}_{01}$, ${\rm Aut}\big( \mathcal{L}%
_{01}\big) $, consists of the invertible matrices of the following form:%
\begin{equation*}
\varphi =%
\begin{pmatrix}
a_{11} & a_{12} & 0 \\ 
a_{21} & a_{22} & 0 \\ 
a_{31} & a_{32} & a_{11}a_{22}-a_{12}a_{21}%
\end{pmatrix}%
.
\end{equation*}

\begin{itemize}
\item $\theta =\eta _{1}$. Then $\big( {\rm S},\cdot _{\theta }\big) $
is associative.

\item $\theta =\eta _{2}$. Then $\big( {\rm S},\cdot _{\theta }\big) $
is associative if and only if $\alpha _{1}=0$. So we assume $\alpha _{1}\neq
0$. Let $\varphi $ be the following automorphism: 
\begin{equation*}
\varphi =%
\begin{pmatrix}
1 & 0 & 0 \\ 
0 & \alpha _{1} & 0 \\ 
0 & \alpha _{2} & \alpha _{1}%
\end{pmatrix}%
.
\end{equation*}%
Then $\theta \ast \varphi =\big( 0, \ \Delta _{11},\ \alpha_{3}\Delta _{12}\big) $. Hence we get the algebras ${\rm S}%
_{01}^{\alpha }$.$\allowbreak $ Furthermore, the algebras ${\rm S}%
_{01}^{\alpha }$ and ${\rm S}_{01}^{\beta }$ are isomorphic if and only
if $\alpha =\beta $.

\item $\theta =\eta _{3}$. Then $\big( {\rm S},\cdot _{\theta }\big) $
is associative if and only if $\alpha _{1}=0$. So we assume $\alpha _{1}\neq
0$. Let $\varphi $ be the following automorphism: 
\begin{equation*}
\varphi =%
\begin{pmatrix}
0 & \alpha _{1} & 0 \\ 
1 & 0 & 0 \\ 
0 & \alpha _{2} & \alpha _{1}%
\end{pmatrix}%
.
\end{equation*}%
Then $\theta \ast \varphi =\big( 0,\ \Delta _{11},\ \alpha_{3}\Delta _{12}\big) $. So we get the algebras ${\rm S}%
_{01}^{\alpha }$.

\item $\theta =\eta _{4}$. Then $\big( {\rm S},\cdot _{\theta }\big) $
is nonassociative. Let $\varphi $ be the the following automorphism:%
\begin{equation*}
\varphi =%
\begin{pmatrix}
\frac{{\bf i}\sqrt{\alpha _{1}}}{\alpha _{2}} & \frac{\alpha
_{1}}{\alpha _{2}} & 0 \\ 
0 & 1 & 0 \\ 
0 & -\frac{\alpha _{1}\alpha _{4}}{\alpha _{2}^{2}} & \frac{{\bf i}\sqrt{\alpha _{1}}}{\alpha _{2}}
\end{pmatrix}%
.
\end{equation*}%
Then $\theta \ast \varphi =\big( 0,\ \Delta _{11}, \ \frac{\alpha
_{1}\alpha _{4}+\alpha _{2}\alpha _{3}}{\alpha _{2}}\Delta
_{12}\big) $. So we get the algebras ${\rm S}_{01}^{\alpha }$.
\end{itemize}

\item 
\underline{${\rm S}^{-}=\mathcal{L}_{03}^{-1}$.} Let $\theta \ =\ \big(
B_{1},B_{2},B_{3}\big) \neq 0$ be an arbitrary element of  ${\rm Z}^{2}\big( \mathcal{L}_{03}^{-1},\mathcal{L}_{03}^{-1}\big) $. Then $\theta \in \big\{ \eta _{1}, \ \eta _{2}\big\} $ where%
\begin{longtable}{lcl}
$\eta _{1} $&$=$&$\big( 2\Delta _{11}, \ \alpha _{1}\Delta_{11}+\Delta _{12}, \ \alpha _{2}\Delta _{11}+\Delta
_{13}\big) ,$ \\
$\eta _{2} $&$=$&$\big( -2\Delta _{11},\ \alpha _{1}\Delta_{11}-\Delta _{12}, \ \alpha _{2}\Delta _{11}-\Delta
_{13}\big) ,$
\end{longtable}%
for some $\alpha _{1},\alpha _{2}\in \mathbb{C}$. The automorphism group of $%
\mathcal{L}_{03}^{-1}$, ${\rm Aut}\big( \mathcal{L}_{03}^{-1}\big) $, consists of the invertible matrices of the following form:%
\begin{equation*}
\begin{pmatrix}
1 & 0 & 0 \\ 
a_{21} & a_{22} & 0 \\ 
a_{31} & 0 & a_{33}%
\end{pmatrix}%
, \ 
\begin{pmatrix}
-1 & 0 & 0 \\ 
a_{21} & 0 & a_{23} \\ 
a_{31} & a_{32} & 0%
\end{pmatrix}%
.
\end{equation*}

\begin{itemize}
\item $\bigskip \theta =\eta _{1}$. Then $\big( {\rm S},\cdot _{\theta
}\big) $ is associative if and only if $\alpha _{1}=\alpha _{2}=0$. So we
assume $\big( \alpha _{1},\alpha _{2}\big) \neq \big( 0,0\big) $.
Suppose first that $\alpha _{1}\neq 0$. Let $\varphi $ be the first of the
following matrices if $\alpha _{2}=0$ or the second if $\alpha _{2}\neq 0$:%
\begin{equation*}
\begin{pmatrix}
1 & 0 & 0 \\ 
0 & \alpha _{1} & 0 \\ 
0 & 0 & 1%
\end{pmatrix}%
, \
\begin{pmatrix}
1 & 0 & 0 \\ 
0 & \alpha _{1} & 0 \\ 
0 & 0 & \alpha _{2}%
\end{pmatrix}%
.
\end{equation*}%
Then \begin{center}$\theta \ast \varphi \in 
\big\{\big( 2\Delta _{11},\ \Delta_{11}+\Delta _{12},\ \Delta _{13}\big) ,\ 
\big(2\Delta _{11},\ \Delta _{11}+\Delta_{12},\ \Delta _{11}+\Delta _{13}\big) \big\}$.
\end{center} So we get the
algebras ${\rm S}_{02}$ and ${\rm S}_{03}$. Assume now $\alpha
_{1}=0 $. Let $\varphi $ be the the following automorphism:%
\begin{equation*}
\begin{pmatrix}
1 & 0 & 0 \\ 
0 & 1 & 0 \\ 
0 & 0 & \alpha _{2}%
\end{pmatrix}%
.
\end{equation*}%
Then $\theta \ast \varphi =\big( 2\Delta _{11},\ \Delta_{12}, \ \Delta _{11}+\Delta _{13}\big) $. Hence we get the
algebra ${\rm S}_{04}$.

\item $\theta =\eta _{2}$. Then $\big( {\rm S},\cdot _{\theta }\big) $
is associative if and only if $\alpha _{1}=\alpha _{2}=0$. So we assume $%
\big( \alpha _{1},\alpha _{2}\big) \neq \big( 0,0\big) $. Suppose
first that $\alpha _{1}\neq 0$. Let $\varphi $ be the first of the following
matrices if $\alpha _{2}=0$ or the second if $\alpha _{2}\neq 0$:%
\begin{equation*}
\begin{pmatrix}
-1 & 0 & 0 \\ 
0 & 0 & \alpha _{1} \\ 
0 & 1 & 0%
\end{pmatrix}%
, \ 
\begin{pmatrix}
-1 & 0 & 0 \\ 
0 & 0 & \alpha _{1} \\ 
0 & \alpha _{2} & 0%
\end{pmatrix}%
\end{equation*}%
Then \begin{center} $\theta \ast \varphi \in \big\{\big( 2\Delta _{11}, \ \Delta_{12},\ \Delta _{11}+\Delta _{13}\big) ,\ \big(
2\Delta _{11},\ \Delta _{11}+\Delta_{12},\ \Delta _{11}+\Delta _{13}\big) \big\}$.\end{center} So we get the
algebras ${\rm S}_{04}$ and ${\rm S}_{03}$. Assume now that $\alpha
_{1}=0$. Let $\varphi $ be the the following automorphism:%
\begin{equation*}
\begin{pmatrix}
-1 & 0 & 0 \\ 
0 & 0 & 1 \\ 
0 & \alpha _{2} & 0%
\end{pmatrix}%
.
\end{equation*}%
Then $\theta \ast \varphi =\big( 2\Delta _{11},\ \Delta_{11}+\Delta _{12},\ \Delta _{13}\big) $. Hence we get the
algebra ${\rm S}_{02}$.
\end{itemize}

\item 
\underline{${\rm S}^{-}=\mathcal{L}_{03}^{0}$.} Let $\theta \ =\ \big(
B_{1},B_{2},B_{3}\big) \neq 0$ be an arbitrary element of   $ 
{\rm Z}^{2}\big( \mathcal{L}_{03}^{0},\mathcal{L}_{03}^{0}\big) $. Then $\theta \in \big\{ \eta _{1},\ \ldots ,\ \eta_{7}\big\} $
where%
\begin{eqnarray*}
\eta _{1} &=&\big( -2\Delta _{11},\ \alpha _{1}\Delta_{11}-\Delta _{12},\ -2\Delta _{13}+\alpha _{2}\Delta_{33}\big) , \\
\eta _{2} &=&\big( -2\Delta _{11}, \ \alpha _{1}\Delta_{11}-\Delta _{12},\ \alpha _{2}\Delta _{33}\big) , \\
\eta _{3} &=&\big( -2\Delta _{11},\ \alpha _{1}\Delta_{11}-\Delta _{12},\ \alpha _{2}\Delta _{11}+\alpha_{3}\Delta _{13}+\frac{\alpha _{3}\big( 2+\alpha _{3}\big) }{%
\alpha _{2}}\Delta _{33}\big) , \\
\eta _{4} &=&\big( 2\Delta _{11},\ \alpha _{1}\Delta_{11}+\Delta _{12},\ 2\Delta _{13}+\alpha _{2}\Delta_{33}\big) , \\
\eta _{5} &=&\big( 2\Delta _{11},\ \alpha _{1}\Delta
_{11}+\Delta _{12},\ \alpha _{2}\Delta _{33}\big) , \\
\eta _{6} &=&\big( 2\Delta _{11},\ \alpha _{1}\Delta
_{11}+\Delta _{12},\ \alpha _{2}\Delta _{11}+\alpha
_{3}\Delta _{13}+\frac{\alpha _{3}\big( \alpha _{3}-2\big) }{%
\alpha _{2}}\Delta _{33}\big) , \\
\eta _{7} &=&\big( 2\alpha _{1}\Delta _{11}+\alpha_{2}\Delta _{13},\ \alpha _{3}\Delta_{11}+\alpha_{1}\Delta _{12}+\alpha _{2}\Delta _{23},\ \frac{1-\alpha
_{1}^{2}}{\alpha _{2}}\Delta _{11}+\alpha _{2}\Delta_{33}\big) _{\alpha _{2}\neq 0},
\end{eqnarray*}%
for some $\alpha _{1},\alpha _{2},\alpha _{3}\in \mathbb{C}$.
The automorphism group of $\mathcal{L}_{03}^{0}$, ${\rm Aut}%
\big( \mathcal{L}_{03}^{0}\big) $, consists of the invertible
matrices of the following form:%
\begin{equation*}
\varphi =%
\begin{pmatrix}
1 & 0 & 0 \\ 
a_{21} & a_{22} & 0 \\ 
a_{31} & 0 & a_{33}%
\end{pmatrix}%
.
\end{equation*}

\begin{itemize}
\item $\theta =\eta _{1}$. Then $\big( {\rm S},\cdot _{\theta }\big) $
is associative if and only if $\alpha _{1}=0$. So we assume $\alpha _{1}\neq
0$. Let $\varphi $ be the first of the following matrices if $\alpha _{2}=0$ or
the second if $\alpha _{2}\neq 0$.%
\begin{equation*}
\begin{pmatrix}
1 & 0 & 0 \\ 
0 & \alpha _{1} & 0 \\ 
0 & 0 & 1%
\end{pmatrix}%
, \ 
\begin{pmatrix}
1 & 0 & 0 \\ 
0 & \alpha _{1} & 0 \\ 
\frac{2}{\alpha _{2}} & 0 & \frac{1}{\alpha _{2}}%
\end{pmatrix}%
\end{equation*}%
Then 
\begin{center}$\theta \ast \varphi \in \big\{\big( -2\Delta _{11},\ 
\Delta_{11}-\Delta _{12}, \ -2\Delta _{13}\big),\ \big(
-2\Delta _{11},\ \Delta _{11}-\Delta
_{12},\ \Delta _{33}\big) \big\}$.
\end{center} So we get the algebras ${\rm S}%
_{05}$ and ${\rm S}_{06}$.

\item $\theta =\eta _{2}$. Then $\big( {\rm S},\cdot _{\theta }\big) $
is associative if and only if $\alpha _{1}=0$. So we assume $\alpha _{1}\neq
0$. Let $\varphi $ be the first of the following matrices if $\alpha _{2}\neq 0$
or the second if $\alpha _{2}=0$.%
\begin{equation*}
\begin{pmatrix}
1 & 0 & 0 \\ 
0 & \alpha _{1} & 0 \\ 
0 & 0 & \frac{1}{\alpha _{2}}%
\end{pmatrix}%
, \
\begin{pmatrix}
1 & 0 & 0 \\ 
0 & \alpha _{1} & 0 \\ 
0 & 0 & 1%
\end{pmatrix}%
.
\end{equation*}%
Then 
\begin{center}$\theta \ast \varphi \in \big\{\big( -2\Delta _{11},\ \Delta_{11}-\Delta _{12},\ \Delta _{33}\big) ,\ \big(
-2\Delta _{11},\ \Delta _{11}-\Delta _{12},\ 0\big) \big\}$%
.\end{center} Hence we get the algebras ${\rm S}_{06}$ and ${\rm S}_{07}$.

\item $\theta =\eta _{3}$. Then $\big( {\rm S},\cdot _{\theta }\big) $
is associative if and only if $\alpha _{1}=0$. So we assume $\alpha _{1}\neq
0$. If $\alpha _{3}\big( \alpha _{3}+2\big) \neq 0$, we choose $\varphi $ to
be the following automorphism:%
\begin{equation*}
\varphi =%
\begin{pmatrix}
1 & 0 & 0 \\ 
0 & \alpha _{1} & 0 \\ 
-\frac{\alpha _{2}}{\alpha _{3}+2} & 0 & \frac{\alpha _{2}}{\alpha
_{3}\big( \alpha _{3}+2\big) }%
\end{pmatrix}%
.
\end{equation*}%
Then $\theta \ast \varphi =\big( -2\Delta _{11},\ \Delta_{11}-\Delta _{12},\ \Delta _{33}\big) $. Therefore we
obtain the algebra ${\rm S}_{06}$. If $\alpha _{3}=0$, we choose $\varphi $
to be the following automorphism:%
\begin{equation*}
\varphi =%
\begin{pmatrix}
1 & 0 & 0 \\ 
0 & \alpha _{1} & 0 \\ 
-\frac{\alpha _{2}}{2} & 0 & 1%
\end{pmatrix}%
.
\end{equation*}%
Then $\theta \ast \varphi =\big( -2\Delta _{11},\ \Delta_{11}-\Delta _{12},\ 0\big) $. So we get the algebra ${\rm S}%
_{07}$. If $\alpha _{3}=-2$, we choose $\varphi $ to be the following
automorphism:%
\begin{equation*}
\varphi =%
\begin{pmatrix}
1 & 0 & 0 \\ 
0 & \alpha _{1} & 0 \\ 
\frac{\alpha _{2}}{2} & 0 & 1%
\end{pmatrix}%
.
\end{equation*}%
Then $\theta \ast \varphi =\big( -2\Delta _{11},\ \Delta
_{11}-\Delta _{12},\ -2\Delta _{13}\big) $. Thus we get the
algebra ${\rm S}_{05}$.

\item $\theta =\eta _{4}$. Then $\big( {\rm S},\cdot _{\theta }\big) $
is associative if and only if $\alpha _{1}=0$. So we assume $\alpha _{1}\neq
0$. Let $\varphi $ be the first of the following matrices if $\alpha _{2}=0$ or
the second if $\alpha _{2}\neq 0$.%
\begin{equation*}
\begin{pmatrix}
1 & 0 & 0 \\ 
0 & \alpha _{1} & 0 \\ 
0 & 0 & 1%
\end{pmatrix}%
, \
\begin{pmatrix}
1 & 0 & 0 \\ 
0 & \alpha _{1} & 0 \\ 
-\frac{2}{\alpha _{2}} & 0 & \frac{1}{\alpha _{2}}%
\end{pmatrix}%
.
\end{equation*}%
Then \begin{center}$\theta \ast \varphi \in \big\{\big( 2\Delta _{11},\ \Delta
_{11}+\Delta _{12},\ 2\Delta _{13}\big),\ \big(
2\Delta _{11},\ \Delta _{11}+\Delta
_{12},\ \Delta _{33}\big) \}$.\end{center} 
So we get the algebras ${\rm S}%
_{08}$ and ${\rm S}_{09}$.

\item $\theta =\eta _{5}$. Then $\big( {\rm S},\cdot _{\theta }\big) $
is associative if and only if $\alpha _{1}=0$. So we assume $\alpha _{1}\neq
0$. Let $\varphi $ be the first of the following matrices if $\alpha _{2}\neq 0$
or the second if $\alpha _{2}=0$.%
\begin{equation*}
\begin{pmatrix}
1 & 0 & 0 \\ 
0 & \alpha _{1} & 0 \\ 
0 & 0 & \frac{1}{\alpha _{2}}%
\end{pmatrix}%
, \ 
\begin{pmatrix}
1 & 0 & 0 \\ 
0 & \alpha _{1} & 0 \\ 
0 & 0 & 1%
\end{pmatrix}%
.
\end{equation*}%
Then \begin{center}
    $\theta \ast \varphi \in \big\{\big( 2\Delta _{11},\ \Delta_{11}+\Delta _{12}, \ \Delta _{33}\big),\ \big(
2\Delta _{11}, \ \Delta _{11}+\Delta _{12},\ 0\big) \big\}$.
\end{center} Hence we get the algebras ${\rm S}_{09}$ and ${\rm S}_{10}$.

\item $\theta =\eta _{6}$. Then $\big( {\rm S},\cdot _{\theta }\big) $
is associative if and only if $\alpha _{1}=0$. So we assume $\alpha _{1}\neq
0$. If $\alpha _{3}\big( \alpha _{3}-2\big) \neq 0$, we choose $\varphi $ to
be the following automorphism:%
\begin{equation*}
\varphi =%
\begin{pmatrix}
1 & 0 & 0 \\ 
0 & \alpha _{1} & 0 \\ 
-\frac{\alpha _{2}}{\alpha _{3}-2} & 0 & \frac{\alpha _{2}}{\alpha
_{3}\big( \alpha _{3}-2\big) }%
\end{pmatrix}%
.
\end{equation*}%
Then $\theta \ast \varphi =\big( 2\Delta _{11},\ \Delta_{11}+\Delta _{12},\ \Delta _{33}\big) $. Therefore we
obtain the algebra ${\rm S}_{09}$. If $\alpha _{3}=0$, we choose $\varphi $
to be the following automorphism:%
\begin{equation*}
\varphi =%
\begin{pmatrix}
1 & 0 & 0 \\ 
0 & \alpha _{1} & 0 \\ 
\frac{\alpha _{2}}{2} & 0 & 1%
\end{pmatrix}%
.
\end{equation*}%
Then $\theta \ast \varphi =\big( 2\Delta _{11},\Delta
_{11}+\Delta _{12},0\big) $. So we get the algebra ${\rm S}%
_{10}$. If $\alpha _{3}=2$, we choose $\varphi $ to be the following
automorphism:%
\begin{equation*}
\varphi =%
\begin{pmatrix}
1 & 0 & 0 \\ 
0 & \alpha _{1} & 0 \\ 
-\frac{\alpha _{2}}{2} & 0 & 1%
\end{pmatrix}%
.
\end{equation*}%
Then $\theta \ast \varphi =\big( 2\Delta _{11},\Delta
_{11}+\Delta _{12},2\Delta _{13}\big) $. Thus we get the
algebra ${\rm S}_{08}$.

\item $\theta =\eta _{7}$. Then $\big( {\rm S},\cdot _{\theta }\big) $
is associative if and only if $\alpha _{3}=0$. So we assume $\alpha _{3}\neq
0$. Let $\varphi $ be the following automorphism:%
\begin{equation*}
\varphi =%
\begin{pmatrix}
1 & 0 & 0 \\ 
0 & \alpha _{3} & 0 \\ 
-\frac{\alpha _{1}}{\alpha _{2}} & 0 & \frac{1}{\alpha _{2}}%
\end{pmatrix}%
.
\end{equation*}%
Then $\theta \ast \varphi =\big( \Delta _{13}, \ \Delta_{11}+\Delta _{23},\ \Delta _{11}+\Delta _{33}\big) 
$. So we obtain the algebras ${\rm S}_{11}$.
\end{itemize}

\item
\underline{${\rm S}^{-}=\mathcal{L}_{03}^{1}$.}  Let $\theta \ =\ \big(
B_{1},B_{2},B_{3}\big) \neq 0$ be an arbitrary element of    $ 
{\rm Z}^{2}\big( \mathcal{L}_{03}^{1},\mathcal{L}_{03}^{1}\big) $. Then $\theta \in \big\{ \eta _{1},\ \eta _{2}\big\} $ where%
\begin{eqnarray*}
\eta _{1} &=&\big( 2\Delta _{11},\ \alpha _{1}\Delta_{11}+\Delta _{12},\ \alpha _{2}\Delta _{11}+\Delta_{13}\big) , \\
\eta _{2} &=&\big( -2\Delta _{11},\ \alpha _{1}\Delta_{11}-\Delta _{12},\ \alpha _{2}\Delta _{11}-\Delta_{13}\big) ,
\end{eqnarray*}%
for some $\alpha _{1},\alpha _{2}\in \mathbb{C}$.
The automorphism group of $\mathcal{L}_{03}^{1}$, ${\rm Aut}%
\big( \mathcal{L}_{03}^{1}\big) $, consists of the invertible
matrices of the following form:%
\begin{equation*}
\varphi =%
\begin{pmatrix}
1 & 0 & 0 \\ 
a_{21} & a_{22} & a_{23} \\ 
a_{31} & a_{32} & a_{33}%
\end{pmatrix}%
.
\end{equation*}

\begin{itemize}
\item $\theta =\eta _{1}$. Let $\varphi $ be the first of the following
matrices if $\alpha _{1}\neq 0$ or the second if $\alpha _{1}=0$.
\end{itemize}

\begin{equation*}
\begin{pmatrix}
1 & 0 & 0 \\ 
0 & \alpha _{1} & 0 \\ 
0 & \alpha _{2} & 1%
\end{pmatrix}%
, \ 
\begin{pmatrix}
1 & 0 & 0 \\ 
0 & 0 & 1 \\ 
0 & \alpha _{2} & 0%
\end{pmatrix}%
.
\end{equation*}%
Then $\theta \ast \varphi =\big( 2\Delta _{11}, \ \Delta_{11}+\Delta _{12},\ \Delta _{13}\big) $. Thus we obtain
the algebra ${\rm S}_{12}$.

\begin{itemize}
\item $\theta =\eta _{2}$. Let $\varphi $ be the first of the following
matrices if $\alpha _{1}\neq 0$ or the second if $\alpha _{1}=0$.
\end{itemize}

\begin{equation*}
\begin{pmatrix}
1 & 0 & 0 \\ 
0 & \alpha _{1} & 0 \\ 
0 & \alpha _{2} & 1%
\end{pmatrix}%
,%
\begin{pmatrix}
1 & 0 & 0 \\ 
0 & 0 & 1 \\ 
0 & \alpha _{2} & 0%
\end{pmatrix}%
.
\end{equation*}%
Then $\theta \ast \varphi =\big( -2\Delta _{11},\ \Delta_{11}-\Delta _{12}, \ -\Delta _{13}\big) $. Thus we obtain
the algebra ${\rm S}_{13}$.
\end{enumerate}

\subsubsection{Corollaries from Theorem A2}

\begin{definition}
An algebra ${\rm A}$ is called  assosymmetric if the following identities hold
\begin{center}
    $\big(
x,y,z\big) =\big( y,x,z\big)=\big( x,z,y\big).$
\end{center}
\end{definition}

In the following result we give a characterization of assosymmetric algebras:
\begin{lemma}\label{charcterization}
     An algebra ${\rm A}$ is assosymmetric if and only if it is semi-alternative and Lie-admissible.
\end{lemma}
   
\begin{proof}
It is known that if ${\rm A}$ is assosymmetric, then it is Lie-admissible. On the other hand, suppose that ${\rm A}$ is semi-alternative and Lie-admissible. Let $x,y,z\in {\rm A}$. Then, from (\ref{id1-semi}), we have    
    $\big(
x,y,z\big) =\big( x,z,y\big).$
Since  ${\rm A}$ is semi-alternative, we have 
    $\big(
y,x,z\big) =\big( x,z,y\big).$
Thus, $(x,y,z)=(y,x,z)=(x,z,y)$ and therefore ${\rm A}$ is assosymmetric. 
\end{proof}
Since all complex $3$-dimensional Malcev algebras are Lie, all complex $3$-dimensional semi-alternative algebras are Lie-admissible. So, by Lemma \ref{charcterization}, we have the following result:
\begin{corollary}\label{assos}
    Let ${\rm A}$ be a complex $3$-dimensional assosymmetric algebra.
Then ${\rm A}$ is isomorphic to one of the semi-alternative algebras listed in Theorem A2.
\end{corollary}
The question arises now: Is there a $4$-dimensional (non-assosymmetric) semi-alternative algebra?
Following Lemma \ref{commutator} and Lemma \ref{charcterization}, we need to consider semi-alternative structures on non-Lie Malcev algebras  of dimension $4$.
\begin{lemma}
\label{non-Lie Malcev}Up to isomorphism, there exists a unique non-Lie
Malcev algebra of dimension $4$, that is%
\begin{longtable}{lcllllllllllllllllll}
${\rm M}$&$:$&
$\left[ e_{1},e_{2}\right] =e_{3}$&
$\left[ e_{1},e_{4}\right]=e_{1}$&
$\left[ e_{2},e_{4}\right]=e_{2}$&
$\left[ e_{3},e_{4}\right] =-e_{3}$
\end{longtable}
\end{lemma}

Note that ${\rm B}_{1}^{-1}$ is a Malcev algebra which is isomorphic to $%
{\rm M}$ by means of the mapping $\varphi :{\rm M}\longrightarrow 
{\rm B}_{1}^{-1}$ given by $\varphi \left( e_{1}\right) =e_{2},$\ 
$\varphi \left(e_{2}\right) =e_{3},$ \ 
$\varphi \left(e_{3}\right) =e_{4},$ \ 
$\varphi \left(e_{4}\right) =-e_{1}$.

\begin{theorem} \label{4-dim non-asso}
Let ${\mathbb S}$ be a complex semi-alternative algebra of dimension $4$. Then ${\mathbb S}$\ is assosymetric or it is isomorphic
to one of the following non-assosymetric algebras:

\begin{longtable}{llllll}

${\mathbb S}_{01}$ & $:$ & $e_{1} e_{2} = e_{3}$ & $e_{2} e_{1} = -e_{3}$ & $e_{1} e_{4} = 2e_{1}$ & $e_{2} e_{4} = 2e_{2}$ \\
 & & $e_{4} e_{3} = 2e_{3}$ & $e_{4} e_{4} = 2e_{4}$ & & \\
${\mathbb S}_{02}$ & $:$ & $e_{1} e_{2} = e_{3}$ & $e_{2} e_{1} = -e_{3}$ & $e_{1} e_{4} = 2e_{1}$ & $e_{2} e_{4} = 2e_{2} + e_{3}$ \\
 & & $e_{4} e_{2} = e_{3}$ & $e_{4} e_{3} = 2e_{3}$ & $e_{4} e_{4} = -4e_{1} + 2e_{4}$ & \\
${\mathbb S}_{03}$ & $:$ & $e_{1} e_{2} = e_{3}$ & $e_{2} e_{1} = -e_{3}$ & $e_{1} e_{4} = 2e_{1}$ & $e_{2} e_{4} = 2e_{2}$ \\
 & & $e_{4} e_{3} = 2e_{3}$ & $e_{4} e_{4} = e_{3} + 2e_{4}$ & & \\
${\mathbb S}_{04}$ & $:$ & $e_{1} e_{2} = e_{3}$ & $e_{2} e_{1} = -e_{3}$ & $e_{4} e_{1} = -2e_{1}$ & $e_{4} e_{2} = -2e_{2}$ \\
 & & $e_{3} e_{4} = -2e_{3}$ & $e_{4} e_{4} = -2e_{4}$ & & \\
${\mathbb S}_{05}$ & $:$ & $e_{1} e_{2} = e_{3}$ & $e_{2} e_{1} = -e_{3}$ & $e_{4} e_{1} = -2e_{1}$ & $e_{2} e_{4} = e_{3}$ \\
 & & $e_{4} e_{2} = -2e_{2} + e_{3}$ & $e_{3} e_{4} = -2e_{3}$ & $e_{4} e_{4} = -4e_{1} - 2e_{4}$ & \\
${\mathbb S}_{06}$ & $:$ & $e_{1} e_{2} = e_{3}$ & $e_{2} e_{1} = -e_{3}$ & $e_{4} e_{1} = -2e_{1}$ & $e_{4} e_{2} = -2e_{2}$ \\
 & & $e_{3} e_{4} = -2e_{3}$ & $e_{4} e_{4} = e_{3} - 2e_{4}$ & & \\
\end{longtable}
 
\end{theorem}

\begin{proof}
Let ${\mathbb S}$ be a complex semi-alternative algebra of dimension four.
Suppose that ${\mathbb S}$ is non-assosymetric. Then, by Lemma \ref{charcterization} and Lemma \ref%
{non-Lie Malcev}, we may assume ${\mathbb S}^{-}={\rm M}$.  Choose an
arbitrary element $\theta =\left( B_{1},B_{2},B_{3},B_{4}\right) \in {\rm Z}^{2}\big( {\rm M},{\rm M}\big) $. Then $\theta \in \left\{
\eta _{1},\eta _{2}\right\} $ where%
\begin{longtable}{lcl}
$\eta _{1} $&$=$&$\big( \Delta _{14}-4\alpha _{1}\Delta _{44},\ 
\Delta_{24}+4\alpha _{2}\Delta _{44},\ 
\alpha _{2}\Delta _{14}+\alpha_{1}\Delta _{24}+\Delta _{34}+\alpha _{3}\Delta _{44},\ 
2\Delta
_{44}\big) ,$ \\
$\eta _{2} $&$=$&$\big( -\Delta _{14}-4\alpha _{1}\Delta_{44},\ 
-\Delta _{24}+4\alpha _{2}\Delta _{44},\ 
\alpha _{2}\Delta_{14}+\alpha _{1}\Delta _{24}-\Delta _{34}+\alpha _{3}\Delta_{44},\
-2\Delta _{44}\big),$
\end{longtable}
\noindent for some $\alpha _{1},\alpha _{2},\alpha _{3}\in \mathbb{C}$. The
automorphism group of ${\rm M}$, $\mathrm{Aut}\left( {\rm M}\right) $%
, consists of the invertible matrices of the following form:%
\begin{equation*}
\varphi =%
\begin{pmatrix}
a_{11} & a_{12} & 0 & a_{14} \\ 
a_{21} & a_{22} & 0 & a_{24} \\ 
\frac{a_{11}a_{24}-a_{21}a_{14}}{2} & \frac{a_{12}a_{24}-a_{22}a_{14}}{2} & 
a_{11}a_{22}-a_{12}a_{21} & a_{34} \\ 
0 & 0 & 0 & 1%
\end{pmatrix}%
.
\end{equation*}

\begin{itemize}
\item $\theta =\eta _{1}$. Let $\varphi =\bigl(a_{ij}\bigr)\in $ $\text{Aut}\left( {\rm M}\right) $.
Then \begin{longtable}{lcl}
$\eta _{1}\ast \varphi $&$=$&$\big( \Delta _{14}-4\beta _{1}\Delta_{44},\ 
\Delta _{24}+4\beta _{2}\Delta _{44},\ 
\beta _{2}\Delta_{14}+\beta _{1}\Delta _{24}+\Delta _{34}+\beta _{3}\Delta_{44},\ 2\Delta _{44}\big),$ 
\end{longtable} where%
\begin{longtable}{lcllcllcl}
$\beta _{1} $&$=$&$\frac{\alpha
_{1}a_{22}+\alpha _{2}a_{12}}{a_{11}a_{22}-a_{12}a_{21}},$&
$\beta _{2} $&$=$&$\frac{ \alpha
_{1}a_{21}+\alpha _{2}a_{11}}{a_{11}a_{22}-a_{12}a_{21}},$&
$\beta _{3} $&$=$&$\frac{\alpha _{3}+4\alpha
_{1}a_{24}+4\alpha _{2}a_{14}}{a_{11}a_{22}-a_{12}a_{21}}.$
\end{longtable}
From here, we may assume $\big( \alpha _{1},\alpha _{2},\alpha _{3}\big)
\ \in \ \big \{\left( 0,0,0\right) ,\ \left( 1,0,0\right) ,\ \left( 0,0,1\right) \big\}$. So
we get the algebras ${\mathbb S}_{01},$ ${\mathbb S}_{02},$ and ${\mathbb S}_{03}$.

\item $\theta =\eta _{2}$.
Let $\varphi =\bigl(a_{ij}\bigr)\in $ $\text{Aut}\left( {\rm M}\right) $.
Then
\begin{longtable}{lcl}
$\eta _{2}\ast \varphi $&$=$&$\big( -\Delta _{14}-4\beta _{1}\Delta
_{44},\ 
-\Delta _{24}+4\beta _{2}\Delta _{44},\ 
\beta _{2}\Delta_{14}+\beta _{1}\Delta _{24}-\Delta _{34}+\beta _{3}\Delta
_{44},\ -2\Delta _{44}\big),$ 
\end{longtable}\noindent where $\beta _{1},\beta _{2},\beta _{3}$ are as given in the case $\theta =\eta _{1}$. So we may assume \begin{center}
    $\big( \alpha _{1},\alpha _{2},\alpha _{3}\big)
\ \in\  \big\{\left( 0,0,0\right) , \ \left( 1,0,0\right) ,\ \left( 0,0,1\right) \big\}$.
\end{center} Thus
we get the algebras ${\mathbb S}_{04},$ \ ${\mathbb S}_{05},$ and ${\mathbb S}_{06}$.
\end{itemize}
\end{proof}

\section{The geometric classification of 
  algebras}

\subsection{Preliminaries: definitions and notation}
Given an $n$-dimensional vector space $\mathbb V$, the set ${\rm Hom}(\mathbb V \otimes \mathbb V,\mathbb V) \cong \mathbb V^* \otimes \mathbb V^* \otimes \mathbb V$ is a vector space of dimension $n^3$. This space has the structure of the affine variety $\mathbb{C}^{n^3}.$ Indeed, let us fix a basis $e_1,\dots,e_n$ of $\mathbb V$. Then any $\mu\in {\rm Hom}(\mathbb V \otimes \mathbb V,\mathbb V)$ is determined by $n^3$ structure constants $c_{ij}^k\in\mathbb{C}$ such that
$\mu(e_i\otimes e_j)=\sum\limits_{k=1}^nc_{ij}^ke_k$. A subset of ${\rm Hom}(\mathbb V \otimes \mathbb V,\mathbb V)$ is {\it Zariski-closed} if it can be defined by a set of polynomial equations in the variables $c_{ij}^k.$ 

Let $T$ be a set of polynomial identities.
The set of algebra structures on $\mathbb V$ satisfying polynomial identities from $T$ forms a Zariski-closed subset of the variety ${\rm Hom}(\mathbb V \otimes \mathbb V,\mathbb V)$. We denote this subset by $\mathbb{L}(T)$.
The general linear group ${\rm GL}(\mathbb V)$ acts on $\mathbb{L}(T)$ by conjugations:
$$ (g * \mu )(x\otimes y) = g\mu(g^{-1}x\otimes g^{-1}y)$$
for $x,y\in \mathbb V$, $\mu\in \mathbb{L}(T)\subset {\rm Hom}(\mathbb V \otimes\mathbb V, \mathbb V)$ and $g\in {\rm GL}(\mathbb V)$.
Thus, $\mathbb{L}(T)$ is decomposed into ${\rm GL}(\mathbb V)$-orbits that correspond to the isomorphism classes of algebras.
Let ${\mathcal O}(\mu)$ denote the orbit of $\mu\in\mathbb{L}(T)$ under the action of ${\rm GL}(\mathbb V)$ and $\overline{{\mathcal O}(\mu)}$ denote the Zariski closure of ${\mathcal O}(\mu)$.

Let $\bf A$ and $\bf B$ be two $n$-dimensional algebras satisfying the identities from $T$, and let $\mu,\lambda \in \mathbb{L}(T)$ represent $\bf A$ and $\bf B$, respectively.
We say that $\bf A$ degenerates to $\bf B$ and write $\bf A\to \bf B$ if $\lambda\in\overline{{\mathcal O}(\mu)}$.
Note that in this case we have $\overline{{\mathcal O}(\lambda)}\subset\overline{{\mathcal O}(\mu)}$. Hence, the definition of degeneration does not depend on the choice of $\mu$ and $\lambda$. If $\bf A\not\cong \bf B$, then the assertion $\bf A\to \bf B$ is called a {\it proper degeneration}. We write $\bf A\not\to \bf B$ if $\lambda\not\in\overline{{\mathcal O}(\mu)}$.

Let $\bf A$ be represented by $\mu\in\mathbb{L}(T)$. Then  $\bf A$ is  {\it rigid} in $\mathbb{L}(T)$ if ${\mathcal O}(\mu)$ is an open subset of $\mathbb{L}(T)$.
 Recall that a subset of a variety is called irreducible if it cannot be represented as a union of two non-trivial closed subsets.
 A maximal irreducible closed subset of a variety is called an {\it irreducible component}.
It is well known that any affine variety can be represented as a finite union of its irreducible components in a unique way.
The algebra $\bf A$ is rigid in $\mathbb{L}(T)$ if and only if $\overline{{\mathcal O}(\mu)}$ is an irreducible component of $\mathbb{L}(T)$.

\medskip

\noindent {\bf Method of the description of degenerations of algebras.} In the present work we use the methods applied to Lie algebras in \cite{GRH}.
First of all, if $\bf A\to \bf B$ and $\bf A\not\cong \bf B$, then $\mathfrak{Der}(\bf A)<\mathfrak{Der}(\bf B)$, where $\mathfrak{Der}(\bf A)$ is the   algebra of derivations of $\bf A$. We compute the dimensions of algebras of derivations and check the assertion $\bf A\to \bf B$ only for such $\bf A$ and $\bf B$ that $\mathfrak{Der}(\bf A)<\mathfrak{Der}(\bf B)$.

To prove degenerations, we construct families of matrices parametrized by $t$. Namely, let $\bf A$ and $\bf B$ be two algebras represented by the structures $\mu$ and $\lambda$ from $\mathbb{L}(T)$ respectively. Let $e_1,\dots, e_n$ be a basis of $\mathbb  V$ and $c_{ij}^k$ ($1\le i,j,k\le n$) be the structure constants of $\lambda$ in this basis. If there exist $a_i^j(t)\in\mathbb{C}$ ($1\le i,j\le n$, $t\in\mathbb{C}^*$) such that $E_i^t=\sum\limits_{j=1}^na_i^j(t)e_j$ ($1\le i\le n$) form a basis of $\mathbb V$ for any $t\in\mathbb{C}^*$, and the structure constants of $\mu$ in the basis $E_1^t,\dots, E_n^t$ are such rational functions $c_{ij}^k(t)\in\mathbb{C}[t]$ that $c_{ij}^k(0)=c_{ij}^k$, then $\bf A\to \bf B$.
In this case  $E_1^t,\dots, E_n^t$ is called a {\it parametrized basis} for $\bf A\to \bf B$.
In  case of  $E_1^t, E_2^t, \ldots, E_n^t$ is a {\it parametric basis} for ${\rm A}\to {\bf B},$ it will be denoted by
${\rm A}\xrightarrow{(E_1^t, E_2^t, \ldots, E_n^t)} {\bf B}$. 
To simplify our equations, we will use the notation $A_i=\langle e_i,\dots,e_n\rangle,\ i=1,\ldots,n$ and write simply $A_pA_q\subset A_r$ instead of $c_{ij}^k=0$ ($i\geq p$, $j\geq q$, $k< r$).

Let ${\rm A}(*):=\{ {\rm A}(\alpha)\}_{\alpha\in I}$ be a series of algebras, and let $\bf B$ be another algebra. Suppose that for $\alpha\in I$, $\bf A(\alpha)$ is represented by the structure $\mu(\alpha)\in\mathbb{L}(T)$ and $\bf B$ is represented by the structure $\lambda\in\mathbb{L}(T)$. Then we say that $\bf A(*)\to \bf B$ if $\lambda\in\overline{\{{\mathcal O}(\mu(\alpha))\}_{\alpha\in I}}$, and $\bf A(*)\not\to \bf B$ if $\lambda\not\in\overline{\{{\mathcal O}(\mu(\alpha))\}_{\alpha\in I}}$.

Let $\bf A(*)$, $\bf B$, $\mu(\alpha)$ ($\alpha\in I$) and $\lambda$ be as above. To prove $\bf A(*)\to \bf B$ it is enough to construct a family of pairs $(f(t), g(t))$ parametrized by $t\in\mathbb{C}^*$, where $f(t)\in I$ and $g(t)\in {\rm GL}(\mathbb V)$. Namely, let $e_1,\dots, e_n$ be a basis of $\mathbb V$ and $c_{ij}^k$ ($1\le i,j,k\le n$) be the structure constants of $\lambda$ in this basis. If we construct $a_i^j:\mathbb{C}^*\to \mathbb{C}$ ($1\le i,j\le n$) and $f: \mathbb{C}^* \to I$ such that $E_i^t=\sum\limits_{j=1}^na_i^j(t)e_j$ ($1\le i\le n$) form a basis of $\mathbb V$ for any  $t\in\mathbb{C}^*$, and the structure constants of $\mu({f(t)})$ in the basis $E_1^t,\dots, E_n^t$ are such rational functions $c_{ij}^k(t)\in\mathbb{C}[t]$ that $c_{ij}^k(0)=c_{ij}^k$, then $\bf A(*)\to \bf B$. In this case  $E_1^t,\dots, E_n^t$ and $f(t)$ are called a parametrized basis and a {\it parametrized index} for $\bf A(*)\to \bf B$, respectively.

We now explain how to prove $\bf A(*)\not\to\mathcal  \bf B$.
Note that if $\mathfrak{Der} \ \bf A(\alpha)  > \mathfrak{Der} \  \bf B$ for all $\alpha\in I$ then $\bf A(*)\not\to\bf B$.
One can also use the following  Lemma, whose proof is the same as the proof of  \cite[Lemma 1.5]{GRH}.

\begin{lemma}\label{gmain}
Let $\mathfrak{B}$ be a Borel subgroup of ${\rm GL}(\mathbb V)$ and ${\rm R}\subset \mathbb{L}(T)$ be a $\mathfrak{B}$-stable closed subset.
If $\bf A(*) \to \bf B$ and for any $\alpha\in I$ the algebra $\bf A(\alpha)$ can be represented by a structure $\mu(\alpha)\in{\rm R}$, then there is $\lambda\in {\rm R}$ representing $\bf B$.
\end{lemma}

\subsection{The geometric classification of   
  algebras}

\subsubsection{The geometric classification of   
$\mathfrak{perm}$ algebras}

\begin{theorem}\label{thm:geo_perm}
The variety of complex $3$-dimensional $\mathfrak{perm}$ algebras has dimension $9$ and it has $4$ irreducible components defined by  
\begin{center}
$\mathcal{C}_1=\overline{\mathcal{O}( {\rm A}_{07})},$ \
$\mathcal{C}_2=\overline{\mathcal{O}( {\rm A}_{18})},$ \
$\mathcal{C}_3=\overline{\mathcal{O}( {\rm A}_{24})},$ \ and
$\mathcal{C}_4=\overline{\mathcal{O}( {\rm A}_{14}^{\alpha})}.$ \
\end{center}
In particular, there are only $3$ rigid algebras in this variety.
 
\end{theorem}
\begin{proof}
After carefully  checking  the dimensions of orbit closures of the more important for us algebras, we have

\begin{longtable}{rclrclrclrcl}
$\dim \mathcal{O}({\rm A}_{07})$&$=$&$9,$ &
$\dim \mathcal{O}({\rm A}_{24})$&$=$&$7,$ & 
$\dim \mathcal{O}({\rm A}_{14}^{\alpha})$&$=$&$ 6,$ &
$\dim \mathcal{O}({\rm A}_{18})$&$=$&$3.$
\end{longtable}
\noindent
${\rm A}_{18}$ is rigid due to~\cite{akks}. 
Thanks to~\cite{akks}, ${\rm A}_{07}$ is rigid in the variety of associative commutative algebras and each commutative associative algebra is in the irreducible component defined by ${\rm A}_{07}$.  
Since ${\rm A}_{07}$ is commutative, we have
${\rm A}_{07} \not\to \big\{ {\rm A}_{24},\  {\rm A}_{14}^{\alpha }  \big\}$. The principal non-degeneration
${\rm A}_{24} \not\to {\rm A}_{14}^{\alpha\neq 0}$ is justified by the following condition
\begin{center} 
$\mathcal R=\big\{ 
c_{22}^3c_{11}^3=c_{12}^3c_{21}^3, \ 
A_3^2=0, \ A_3A_1 \subseteq A_3, \ A_2A_1 \subseteq A_2
\big\}.
$\end{center}
All necessary degenerations are given below 

\begin{longtable}{|lcl|lcl|}
\hline
${\rm A}_{13} $&$\xrightarrow{ (te_1, \ e_2+e_3,\  -te_3)} $&${\rm A}_{12}$ & 
    ${\rm A}^{-it^{-1}}_{14} $&$\xrightarrow{ (te_1+te_2-4t^{-2}e_3, \ t^5e_2, \ -i t^2e_1+i t^2e_2)} $&${\rm A}_{13}$\\\hline  
    ${\rm A}_{24} $&$\xrightarrow{ (e_1+e_2,\ t^2e_2+e_3,\ te_1)} $&${\rm A}_{16}$ &
${\rm A}_{24} $&$\xrightarrow{ (e_1, \ e_3,\ te_2)} $&${\rm A}_{21}$ \\
 
\hline
\end{longtable}

\end{proof}

\subsubsection{The geometric classification of   
associative  algebras}

\begin{theorem}\label{thm:geo_assoc}
The variety of complex $3$-dimensional associative algebras has dimension $9$ and it has $8$ irreducible components defined by  
\begin{center}
$\mathcal{C}_1=\overline{\mathcal{O}( {\rm A}_{07})},$ \
$\mathcal{C}_2=\overline{\mathcal{O}( {\rm A}_{17})},$ \
$\mathcal{C}_3=\overline{\mathcal{O}( {\rm A}_{18})},$ \
$\mathcal{C}_4=\overline{\mathcal{O}( {\rm A}_{19})}.$ \\
$\mathcal{C}_5=\overline{\mathcal{O}( {\rm A}_{20})},$ \
$\mathcal{C}_6=\overline{\mathcal{O}( {\rm A}_{23})},$ \
$\mathcal{C}_7=\overline{\mathcal{O}( {\rm A}_{24})},$ \
$\mathcal{C}_8=\overline{\mathcal{O}( {\rm A}_{14}^{\alpha})},$ 
\end{center}
In particular, there are only $7$ rigid algebras in this variety.
\end{theorem}

\begin{proof}
Thanks to Theorem \ref{thm:geo_perm}, 
we have the list of algebras that define irreducible components in the variety of $\mathfrak{perm}$ algebras. Hence, below we consider only this list of algebras and non-$\mathfrak{perm}$ associative algebras from Theorem \ref{3-dim assoc}. After carefully  checking  the dimensions of orbit closures of the more important for us algebras, we have 
\begin{longtable}{rcrcrcrcl} 
&&&&$\dim\mathcal{O}({\rm A}_{07})$&$=$&$9,$ \\ 

$\dim\mathcal{O}({\rm A}_{20})$&$=$&$\dim\mathcal{O}({\rm A}_{23})$&$=$&$\dim\mathcal{O}({\rm A}_{24})$&$=$&$7,$ \\ 

&&&&$\dim\mathcal{O}({\rm A}_{14}^{\alpha})
$&$=$&$6,$ \\ 

&&&&$\dim\mathcal{O}({\rm A}_{19})$&$=$&$5,$ \\

&&$\dim\mathcal{O}({\rm A}_{17})$&$=$&$\dim\mathcal{O}({\rm A}_{18})$&$=$&$3$. \\
\end{longtable}
\noindent

${\rm A}_{17}$ and ${\rm A}_{18}$ are rigid due to \cite{akks}. Thanks to~\cite{akks}, ${\rm A}_{07}$ is rigid in the variety of associative commutative algebras and each commutative associative algebra is in the irreducible component defined by ${\rm A}_{07}$. Algebras ${\rm A}_{07},$ ${\rm A}_{14}^{\alpha},$  ${\rm A}_{18}$, and ${\rm A}_{24}$
are $\mathfrak{perm}$. Hence,
\begin{center}
    
$\big\{{\rm A}_{07},\ {\rm A}_{14}^{\alpha},\ {\rm A}_{18},  \
{\rm A}_{24} \big\}$
$\not\to$  
$\big \{
{\rm A}_{17}, \ 
{\rm A}_{19}, \ 
{\rm A}_{20},\ 
{\rm A}_{23}
\big\}.$ 
\end{center}

Below we list all the reasons for necessary non-degenerations. 

\begin{longtable}{lcl|l}
\hline

${\rm A}_{20}$ & $\not\rightarrow$ & 

$\begin{array}{llll}
 {\rm A}_{14}^{\alpha}, \ {\rm A}_{19} \\
\end{array}$ 
& 
$\mathcal R=\left\{\begin{array}{lllllllll}
c_{12}^2=c_{21}^2=c_{31}^3, \
c_{21}^3=0, \ 
c_{22}^3=0, \\
A_1A_2+A_2A_1 \subseteq A_2,\ 
A_1A_3+A_3A_1 \subseteq A_3, \ 
A_2A_3=0\ 
\end{array}\right\}
$\\

\hline


${\rm A}_{23}$ & $\not\rightarrow$ & 

$\begin{array}{llll}
 {\rm A}_{14}^{\alpha},\  {\rm A}_{19} \\
\end{array}$ 
& 
$\mathcal R=\left\{\begin{array}{lllllll}
c_{11}^3c_{22}^3=c_{12}^3c_{21}^3, \
c_{12}^2=c_{21}^2, \ 
c_{11}^1=c_{13}^3, \\

A_1A_2 +A_2A_1 \subseteq A_2, \
A_1A_3 +A_3A_1 \subseteq A_3, \
A_3^2=0

\end{array}\right\}
$\\

\hline

\end{longtable}

All necessary degenerations are  
${\rm A}_{23}  \xrightarrow{ (e_1+e_2, \ e_3, \ te_2)} {\rm A}_{15}$ and 
${\rm A}_{23} \xrightarrow{ (e_1,\ e_3,\ te_2)}  {\rm A}_{22}.$

\end{proof}

\subsubsection{The geometric classification of   
right alternative  algebras}

\begin{theoremG1}\label{thm:geo_rightalt}
The variety of complex $3$-dimensional right alternative algebras has dimension $9$ and it has $9$ irreducible components defined by  
\begin{center}
$\mathcal{C}_1=\overline{\mathcal{O}( {\rm A}_{07})},$ \
$\mathcal{C}_2=\overline{\mathcal{O}( {\rm A}_{17})},$ \
$\mathcal{C}_3=\overline{\mathcal{O}( {\rm A}_{18})},$ \
$\mathcal{C}_4=\overline{\mathcal{O}( {\rm A}_{19})}.$ \\
$\mathcal{C}_5=\overline{\mathcal{O}( {\rm A}_{20})},$ \
$\mathcal{C}_6=\overline{\mathcal{O}( {\rm A}_{23})},$ \
$\mathcal{C}_7=\overline{\mathcal{O}( {\rm A}_{24})},$ \
$\mathcal{C}_8=\overline{\mathcal{O}( {\rm R}_{09})},$ \ and
$\mathcal{C}_9=\overline{\mathcal{O}( {\rm R}_{10})}.$ 
\end{center}
In particular, there are only $9$ rigid algebras in this variety.
\end{theoremG1}

\begin{proof}
Thanks to Theorem \ref{thm:geo_assoc}, we have the list of algebras that define irreducible components in the variety of associative algebras. Hence, below we consider only this list of algebras and non-associative  right alternative algebras from Theorem A1.
After carefully  checking  the dimensions of orbit closures of the more important for us algebras, we have 
\begin{longtable}{rcrcrcl} 
&&&&$\dim\mathcal{O}({\rm A}_{07})$&$=$&$9,$ \\ 

&&$\dim\mathcal{O}({\rm R}_{09})$&$=$&$\dim\mathcal{O}({\rm R}_{10})$&$=$&$8,$ \\ 

$\dim\mathcal{O}({\rm A}_{20})$&$=$&$\dim\mathcal{O}({\rm A}_{23})$&$=$&$\dim\mathcal{O}({\rm A}_{24})
$&$=$&$7,$ \\ 

&&&&$\dim\mathcal{O}({\rm A}_{14}^{\alpha})$&$=$&$6,$ \\ 

&&&&$\dim\mathcal{O}({\rm A}_{19})$&$=$&$5,$ \\

&&$\dim\mathcal{O}({\rm A}_{17})$&$=$&$\dim\mathcal{O}({\rm A}_{18})$&$=$&$3$. \\
\end{longtable}
\noindent

${\rm A}_{17}$ and ${\rm A}_{18}$ are rigid due to \cite{akks}. Thanks to~\cite{akks}, ${\rm A}_{07}$ is rigid in the variety of associative commutative algebras and each commutative associative algebra is in the irreducible component defined by ${\rm A}_{07}$. Moreover, ${\rm A}_{07} \not\to \big\{
{\rm R}_{09},
{\rm R}_{10}
\big \}.$

Below we list all the important reasons for necessary non-degenerations. 

\begin{longtable}{lcl|l}
\hline



${\rm R}_{09}$ & $\not\rightarrow$ & 

$\begin{array}{llll}
 {\rm A}_{19},\ {\rm A}_{20},\ {\rm A}_{23},\ {\rm A}_{24} \\
\end{array}$ 
& 
$\mathcal R=\left\{\begin{array}{lllllllll}
c_{11}^1=c_{21}^2, \ 
c_{11}^2=0, \
c_{21}^1=0, \
c_{33}^3=0, \
A_1A_2 \subseteq A_3 
\end{array}\right\}
$\\
\hline

 ${\rm R}_{10}$ & $\not\rightarrow$ & 

$\begin{array}{llll}
 {\rm A}_{19},\ {\rm A}_{20},\ {\rm A}_{23},\ {\rm A}_{24} \\
\end{array}$ 
& 
$\mathcal R=\left\{\begin{array}{lllllllll}
c_{11}^1=c_{12}^2=c_{13}^3=c_{31}^3, \ 
A_2A_1 \subseteq A_3,\ 
A_3A_2=0


\end{array}\right\}
$\\
\hline

\end{longtable}

All necessary degenerations are given below 

\begin{longtable}{|lclclcl|}
\hline
    ${\rm R}_{10} $&$\xrightarrow{ (e_1+e_2-e_3,\ te_2-2te_3,\ te_3)} $&${\rm R}_{04}$  & \vrule &
${\rm R}_{09} $&$\xrightarrow{ (e_1+e_2+e_3,\ te_2+2te_3,\ te_3)} $&${\rm R}_{13}$\\
\hline
 &\multicolumn{5}{c}{${\rm R}_{09} \xrightarrow{ (te_1+t^{-1}(1+\alpha) e_2+t^{-3}(1+\alpha)^2e_3,\ te_2+2\alpha t^{-1}e_3,\ e_3)} {\rm A}_{14}^{\alpha}$} &\\
\hline
\end{longtable}

\end{proof}

\subsubsection{The geometric classification of   
semi-alternative algebras}

\begin{theoremG2}\label{thm:geo_semialt}
The variety of complex $3$-dimensional semi-alternative algebras has dimension $9$ and it has $7$ irreducible components defined by  
\begin{center}
$\mathcal{C}_1=\overline{\mathcal{O}( {\rm A}_{07})},$ \
$\mathcal{C}_2=\overline{\mathcal{O}( {\rm S}_{03})},$ \
$\mathcal{C}_3=\overline{\mathcal{O}( {\rm S}_{06})},$ \\
$\mathcal{C}_4=\overline{\mathcal{O}( {\rm S}_{09})},$ \
$\mathcal{C}_5=\overline{\mathcal{O}( {\rm S}_{11})},$ \
$\mathcal{C}_6=\overline{\mathcal{O}( {\rm S}_{12})},$ \ and
$\mathcal{C}_7=\overline{\mathcal{O}( {\rm S}_{13})}.$ 
\end{center}
In particular, there are only $7$ rigid algebras in this variety.
\end{theoremG2}

\begin{proof}

After carefully  checking  the dimensions of orbit closures of the more important for us algebras, we have 
\begin{longtable}{rcrcrcl} 
&&&&$\dim\mathcal{O}({\rm A}_{07})$&$=$&$9,$ \\ 

$\dim\mathcal{O}({\rm S}_{06})$&$=$&$\dim\mathcal{O}({\rm S}_{09})$&$=$&$\dim\mathcal{O}({\rm S}_{11})
$&$=$&$8,$ \\ 

&&&&$\dim\mathcal{O}({\rm S}_{03})$&$=$&$7,$ \\

&&$\dim\mathcal{O}({\rm S}_{12})$&$=$&$\dim\mathcal{O}({\rm S}_{13})$&$=$&$5$. \\ 
\end{longtable}

It is easy to see that the algebras ${\rm S}_{06},$ ${\rm S}_{12}$ are opposite to ${\rm S}_{09},$ ${\rm S}_{13}$, respectively. Moreover, the algebra ${\rm S}_{03}$ is self-opposit. Thus, it is obvious that 
$\{{\rm S}_{06},\ {\rm S}_{09} \} \not\rightarrow {\rm S}_{03}$. 
The two useful non-degenerations are given below:

\begin{longtable}{lcl|l}
\hline
    \multicolumn{4}{c}{Non-degenerations reasons} \\
 
\hline


${\rm S}_{11}$ & $\not\rightarrow$ & 

$  {\rm S}_{03} $ 
& 

$\mathcal R=\left\{\begin{array}{lllllllll}
A_3^2=0, \
A_3A_1\subseteq A_3, \
c_{22}^2=-c_{21}^1, \
2c_{22}^2c_{11}^2=(c_{13}^3)^2+(c_{31}^3)^2\\
\mbox{new basis for }{\rm S}_{11}:
 \mathcal{B} =\big\{f_1=e_1, f_2=e_3, f_3=e_2\big\}
\end{array}\right\}
 $\\
 
\hline

${\rm S}_{06}$ & $\not\rightarrow$ & 

$ {\rm S}_{12},\  {\rm S}_{13}  
$ 
& 

$\mathcal R=\left\{\begin{array}{lllllllll}
A_1A_2+A_2A_1 \subseteq A_2, \
c_{12}^2=c_{21}^2,\\
\mbox{new basis for }{\rm S}_{06}:
 \mathcal{B} =\big\{f_1=e_3, f_2=e_1, f_3=e_2\big\}
\end{array}\right\}
 $\\
 
\hline

\end{longtable}
The last non-degeneration 
${\rm S}_{09} \not\rightarrow  \{ {\rm S}_{12},\  {\rm S}_{13}  \}$ is true due to opposite algebras [${\rm S}_{06},$ ${\rm S}_{12}$ are opposite to ${\rm S}_{09},$ ${\rm S}_{13}$, respectively]. 

All necessary degenerations are given below 

\begin{longtable}{|lcl | lcl|}
\hline
${\rm S}_{01}^{\alpha^{-1}} $&$\xrightarrow{ (\alpha te_1 +te_2,\ e_2,\ te_3)} $&${\rm A}_{14}^\alpha$  & 
${\rm S}_{03} $&$\xrightarrow{ ( \frac 12e_1,\ t^{-1}e_2,\ t^{-1}e_3)} $&${\rm A}_{19}$ \\ \hline

${\rm S}_{09} $&$\xrightarrow{ ( \frac 12e_1,\ e_3,\ t^{-1}e_2)} $&${\rm A}_{23}$ & 
${\rm S}_{06} $&$\xrightarrow{ (\frac 1 2e_1,\ e_3,\ t^{-1}e_2)} $&${\rm A}_{24}$  \\ \hline

${\rm S}_{03} $&$\xrightarrow{ (e_1,\ e_2,\ t^{-1}e_3)} $&${\rm S}_{02}$ &
${\rm S}_{03} $&$\xrightarrow{ (e_1, \ t^{-1}e_2+t(t^2-2)^{-1}e_3,\ 2(2-t^2)^{-1}e_3)}$&$ {\rm S}_{04}$\\\hline

${\rm S}_{06} $&$\xrightarrow{ (e_1+2e_3,\ e_2+t^2e_3,\ te_3)} $&${\rm S}_{05}$  &  
${\rm S}_{06} $&$\xrightarrow{ (e_1, \ e_2,\ te_3)} $&${\rm S}_{07}$  \\ \hline

${\rm S}_{09} $&$\xrightarrow{ (e_1+2e_3,\ e_2,\ te_3)} $&${\rm S}_{08}$ &  
${\rm S}_{09} $&$\xrightarrow{ (e_1,\ e_2,\ te_3)} $&${\rm S}_{10}$ \\

\hline

${\rm S}_{12} $&$\xrightarrow{ (\frac 12 e_1, t^{-1}e_2, e_3)} $&${\rm A}_{17}$ &  
${\rm S}_{13} $&$\xrightarrow{ (\frac 12 e_1, t^{-1}e_2, e_3)} $&${\rm A}_{18}$ \\

\hline

\multicolumn{6}{|c|}{${\rm S}_{11} \xrightarrow{ (\frac 1 2e_1+ \frac 18 e_2+ \frac 1 2e_3, \ 
-\frac 12e_1+  \frac 14e_2+ \frac 1 2e_3, \ \frac3 8t^{-1}e_2)} {\rm A}_{20}$}  
\\
\hline
\multicolumn{6}{|c|}{${\rm S}_{11} \xrightarrow{ ((1+\alpha)^{-1}te_1+(1+\alpha)^{-1}\alpha te_3, \ 
(1+\alpha)^{-2}t^2e_2+(1+\alpha)^{-1}(1-\alpha)t^2e_3, \ (1+\alpha)^{-3}t^3e_2)} {\rm S}_{01}^{\alpha\neq -1}$} 
\\
\hline
\end{longtable}

\end{proof}

\printbibliography 
 
\end{document}